\documentclass [a4paper, 12pt, reqno]{amsart}
\usepackage [latin1]{inputenc}
\usepackage {a4}
\usepackage{amscd}
\usepackage{epsfig}
\usepackage{amssymb}
\usepackage{amsmath}
\usepackage{amsthm}
\usepackage[T1]{fontenc}
\usepackage{ae,aecompl}
\usepackage[arrow, matrix, curve]{xy}
\usepackage{color}

\newcommand{\R} {\ensuremath{\mathbb{R}}}

\newcommand{\C} {\ensuremath{\mathbb{C}}}

\newcommand{\OO}{\mathcal{O}}
\renewcommand{\o}[1]{\overline{#1}}

\newcommand{\dq}{\overline{\partial}}
\newcommand{\wt}[1]{\widetilde{#1}}
\DeclareMathOperator{\Reg}{Reg}
\DeclareMathOperator{\Sing}{Sing}
\DeclareMathOperator{\Hom}{Hom}

\DeclareMathOperator{\Dom}{Dom}

\newtheorem {satz} {Satz} [section]
\newtheorem {lem} [satz] {Lemma}

\newtheorem {thm} [satz] {Theorem}

\DeclareMathOperator{\supp}{supp}

\renewcommand{\Im}{\mbox{Im }}
\newcommand{\coker}{\mbox{coker }}

\renewcommand{\theta}{\vartheta}

\title[$L^2$-theory for $\dq$ on complex spaces] 
{$L^2$-theory for the $\dq$-operator on complex spaces with isolated singularities}

\author{J. Ruppenthal}

\address{Department of Mathematics, University of Wuppertal, Gau{\ss}str. 20, 42119 Wuppertal, Germany.}
\email{ruppenthal@uni-wuppertal.de}

\date{\today}



\begin{document}

\begin{abstract}
We present a refined, improved $L^2$-theory for the $\dq$-operator for $(0,q)$ and $(n,q)$-forms
on Hermitian complex spaces of pure dimension $n$ with isolated singularities.
The general philosophy is to use a resolution of singularities
to obtain a regular model of the $L^2$-cohomology.
\end{abstract}

\maketitle

\section{Introduction}
The present paper is the second part of an attempt to create a systematic $L^2$-theory for the
$\dq$-operator on singular complex spaces. By a refinement and completion of the techniques
introduced in the first part \cite{eins}, we obtain a better picture for
$(0,q)$ and $(n,q)$-forms on a Hermitian complex space $X$ which is of pure dimension $n$ and has 
only isolated singularities.
The general philosophy is to use a resolution of singularities
to obtain a regular model of the $L^2$-cohomology.
This complements also the insights of {\O}vrelid and Vassiliadou in \cite{OvVa4}.

\medskip
The key element of our theory is a new kind of canonical sheaf on $X$ introduced in \cite{eins}
which we denote here by $\mathcal{K}_X^s$. It is the sheaf of germs of holomorphic square-integrable
$n$-forms which satisfy a Dirichlet boundary condition at the singular set $\Sing X$.
It comes as the kernel of the $\dq_s$-operator on square-integrable $(n,0)$-forms (see \eqref{eq:KXs}).
The $\dq_s$-operator is a localized version of the $L^2$-closure of the $\dq$-operator
acting on forms with support away from the singular set (see Section \ref{sec:review}
for the precise definition). If $X$ is a Hermitian complex space with only isolated singularities,
we showed in \cite{eins} that the $\dq_s$-complex
\begin{eqnarray}\label{eq:intro1}
0\rightarrow \mathcal{K}_X^s \hookrightarrow \mathcal{F}^{n,0} \overset{\dq_s}{\longrightarrow}
\mathcal{F}^{n,1} \overset{\dq_s}{\longrightarrow} \mathcal{F}^{n,2} \overset{\dq_s}{\longrightarrow} ... \longrightarrow \mathcal{F}^{n,n}
\rightarrow 0
\end{eqnarray}
is a fine resolution of $\mathcal{K}_X^s$, where the $\mathcal{F}^{n,q}$ are the sheaves of germs of $L^2$-forms
in the domain of the $\dq_s$-operator (see Theorem \ref{thm:exactness2} below).

\medskip
Besides that, we were able to give the following nice representation of $\mathcal{K}_X^s$
in terms of a resolution of singularities:

\begin{thm}\label{thm:Ks}{\bf (\cite{eins}, Theorem 1.10)}
Let $X$ be a Hermitian complex space of pure dimension with only isolated singularities.
Then there exists a resolution of singularities $\pi: M\rightarrow X$ with only normal crossings 
and an effective divisor $D\geq Z-|Z|$ with support on the exceptional set such that:
\begin{eqnarray}\label{eq:fix6}
\mathcal{K}_X^s \cong \pi_* \big( \mathcal{K}_M \otimes \OO(-D) \big),
\end{eqnarray}
where $\mathcal{K}_X^s$ is the canonical sheaf for the $\dq_s$-operator,
$\mathcal{K}_M$ is the usual canonical sheaf on $M$ and $Z=\pi^{-1}(\Sing X)$
the unreduced exceptional divisor.

If the exceptional set of the resolution $\pi: M\rightarrow X$ has only double self-intersections,
which is particularly the case if $\dim X=2$, 
then one can take $D=Z-|Z|$ in \eqref{eq:fix6}.
\end{thm}

By Grauert's direct image theorem \cite{Gr2}, this yields particularly that $\mathcal{K}_X^s$ is a coherent analytic sheaf.
Note that here $\pi: M\rightarrow X$ is not only a resolution of $X$,
but -- in some sense -- also a resolution of $\mathcal{K}_X^s$, making it locally free.
It is moreover not hard to show that
\begin{eqnarray*}
\pi_* \big(\mathcal{K}_M \otimes \OO(-Z) \big) \subset \mathcal{K}_X^s.
\end{eqnarray*}
This follows directly from the proof of Lemma 6.1 in \cite{eins}
and makes it reasonable to conjecture that one can always take $D=Z-|Z|$ in \eqref{eq:fix6}
(as we also know that this is possible e.g. if $\dim X=2$ or if the singularities are homogeneous so that they can be resolved
by a single blow-up).
In the present paper, we show how the $L^2$-theory from \cite{eins} can be refined if we assume that actually $D=Z-|Z|$.\footnote{
We will show e.g. how certain isomorphisms on cohomology from \cite{eins} can be realized explicitly by pull-backs or push-forwards
of differential forms.}
So, assume from now on that this is the case.

\medskip
Assume also for a moment 
that the invertible sheaf $\mathcal{K}_M\otimes\OO(|Z|-Z)$
is locally semi-positive with respect to the base space $X$, i.e. that 
any point $x\in X$ has a small neighborhood $U_x$ such that $\mathcal{K}_M\otimes\OO(|Z|-Z)$
is semi-positive on $\pi^{-1}(U_x)$ in the sense that the holomorphic line bundle associated
to $\mathcal{K}_M\otimes\OO(|Z|-Z)$ is semi-positive on $\pi^{-1}(U_x)$.
Under this assumption Takegoshi's vanishing theorem (see \cite{Ta}, Theorem I and Remark 2(a))
yields the vanishing of the higher direct image sheaves
$$R^q\pi_* \big(\mathcal{K}_M\otimes\OO(|Z|-Z)\big) = 0\ ,\  q\geq 1.$$
So, if 
\begin{eqnarray}\label{eq:intro2}
0\rightarrow \mathcal{K}_M\otimes\OO(|Z|-Z) \hookrightarrow \mathcal{G}^{n,0} \overset{\dq}{\longrightarrow}
\mathcal{G}^{n,1} \overset{\dq}{\longrightarrow} \mathcal{G}^{n,2} \longrightarrow ...
\end{eqnarray}
is a fine resolution of $\mathcal{K}_M\otimes\OO(|Z|-Z)$, then the direct image complex $\pi_* \mathcal{G}^{n,*}$
is another fine resolution of $\mathcal{K}_X^s$ by use of the Leray spectral sequence.
This means that the cohomology of the $L^2$-complex \eqref{eq:intro1} is canonically
isomorphic to the cohomology of the complex \eqref{eq:intro2}.
In other words, \eqref{eq:intro2} is a smooth model for the $L^2$-cohomology of the $\dq_s$-complex \eqref{eq:intro1}.
By use of the $L^2$-version of Serre duality, this also leads to a smooth realization of the
$L^2$-cohomology with respect to the $\dq$-operator in the sense of distributions for $(0,q)$-forms
(see \cite{eins}, Theorem 1.6).

\bigskip
In \cite{eins}, we showed how the semi-positivity condition can be dropped by use of a more sophisticated use of the Leray spectral sequence.
This had the disadvantage that there is no explicit realization of the mappings on cohomology that occur.
However, if $D=Z-|Z|$ in \eqref{eq:fix6}, then we can refine the techniques from \cite{eins} and obtain
the first main result of the present paper:

\begin{thm}\label{thm:main1}
Let $X$ be a Hermitian complex space of pure dimension $n\geq 2$ with only isolated singularities
and $\pi: M\rightarrow X$ a resolution of singularities with only normal crossings such that
$$\mathcal{K}_X^s \cong \pi_* \big( \mathcal{K}_M \otimes \OO(|Z|-Z) \big),$$
where $\mathcal{K}_X^s$ is the canonical sheaf for the $\dq_s$-operator
(i.e. the canonical sheaf of holomorphic $(n,0)$-forms with Dirichlet boundary condition),
$\mathcal{K}_M$ is the usual canonical sheaf on $M$ and $Z=\pi^{-1}(\Sing X)$
the unreduced exceptional divisor.

Then the pull-back of forms under $\pi$ induces for $p\geq 1$ natural exact sequences 
\begin{eqnarray}\label{eq:intro3}
0 \rightarrow H^p(X,\mathcal{K}_X^s) &\overset{[\pi^*_p]}{\longrightarrow}& H^p(M, \mathcal{K}_M\otimes\OO(|Z|-Z))
\longrightarrow \Gamma(X,\mathcal{R}^p) \rightarrow 0,\\
0 \rightarrow H^p_{cpt}(X,\mathcal{K}_X^s) &\overset{[\pi^*_p]}{\longrightarrow}& H^p_{cpt}(M, \mathcal{K}_M\otimes\OO(|Z|-Z))
\longrightarrow \Gamma(X,\mathcal{R}^p) \rightarrow 0,\label{eq:intro4}
\end{eqnarray}
where $\mathcal{R}^p$ is the higher direct image sheaf $R^p\pi_* (\mathcal{K}_M\otimes\OO(|Z|-Z))$.
\end{thm}

The proof of Theorem \ref{thm:main1} is based on the following observation.
If $\mathcal{G}^{n,*}$ is a fine resolution as in \eqref{eq:intro2},
then the non-exactness of the direct image complex $\pi_*\mathcal{G}^{n,*}$
can be expressed by the higher direct image sheaves $\mathcal{R}^p=R^p\pi_* \mathcal{K}_M\otimes \OO(|Z|-Z)$, $p\geq1$.
These are skyscraper sheaves for $X$ has only isolated singularities. 
So, they are acyclic. On the other hand, global sections in $\mathcal{R}^p$ can be expressed
globally by $L^2$-forms with compact support. These two properties allow to express
the cohomology of the canonical sheaf $\mathcal{K}_X^s$ in terms of the cohomology of the direct image complex $\pi_*\mathcal{G}^{n,*}$
modulo global sections in $\mathcal{R}^p$.
Another difficulty is to show that the exact sequences in Theorem \ref{thm:main1} are actually
induced by the pull-back of $L^2$-forms (see Lemma \ref{lem:inclusion}).
We will also see that the surjections in \eqref{eq:intro3} and \eqref{eq:intro4} are simply induced by taking germs of differential forms.
So, we obtain an explicit and localized version of Theorem 1.11 in \cite{eins}.

\bigskip
Another objective of the present paper is to give an $L^2$-version of Theorem \ref{thm:main1}.
By exploiting the fact that the $L^2$- and the $L^2_{loc}$-Dolbeault cohomology are naturally isomorphic
on strongly pseudoconvex domains in complex manifolds,
we are able to achieve that.
This is pretty interesting because it allows to carry over our results to $(0,q)$-forms by the use of the $L^2$-version of Serre duality.
If $N$ is any Hermitian complex manifold, let
$$\dq_{cpt}: A^{p,q}_{cpt}(N) \rightarrow A^{p,q+1}_{cpt}(N)$$
be the $\dq$-operator on smooth forms with compact support in $N$.
Then we denote by
$$\dq_{max}: L^{p,q}(N) \rightarrow L^{p,q+1}(N)$$
the maximal and by
$$\dq_{min}: L^{p,q}(N) \rightarrow L^{p,q+1}(N)$$
the minimal closed Hilbert space extension of the operator $\dq_{cpt}$
as densely defined operator from $L^{p,q}(N)$ to $L^{p,q+1}(N)$.
Let $H^{p,q}_{max}(N)$ be the $L^2$-Dolbeault cohomology on $N$ with respect
to the maximal closed extension $\dq_{max}$, i.e. the $\dq$-operator in the sense of distributions on $N$,
and $H^{p,q}_{min}(N)$ the $L^2$-Dolbeault cohomology with respect to the minimal closed
extension $\dq_{min}$.

We have the following $L^2$-version of the exact sequence \eqref{eq:intro4} in Theorem \ref{thm:main1}:

\begin{thm}\label{thm:main19a}
Let $(X,h)$ be a Hermitian complex space of pure dimension $n\geq 2$ with only isolated singularities and $\pi: M\rightarrow X$ a resolution of singularities
with only normal crossings such that $\mathcal{K}_X^s=\pi_*\big( \mathcal{K}_M\otimes \OO(|Z|-Z)\big)$ as above.

Let $0\leq p < n$, $\Omega\subset\subset X$ a relatively compact domain,
$\wt{\Omega}:=\pi^{-1}(\Omega)$ and $\Omega^*=\Omega-\Sing X$.
Provide $\wt{\Omega}$ with a (regular) Hermitian metric which is equivalent to $\pi^* h$
close to the boundary $b\wt{\Omega}$.

Let $Z:=\pi^{-1}(\Omega\cap \Sing X)$ denote now just the unreduced exceptional divisor over $\Omega$
and let $L_{|Z|-Z}\rightarrow M$ be a Hermitian holomorphic line bundle such that holomorphic sections in $L_{|Z|-Z}$
correspond to holomorphic sections in $\OO(|Z|-Z)$.

Then the pull-back of forms $\pi^*$ induces a natural injective homomorphism
\begin{eqnarray*}
h_p: H^{n,p}_{min}(\Omega^*) \longrightarrow H^{n,p}_{min}(\wt{\Omega},L_{|Z|-Z}),
\end{eqnarray*}
with $\coker h_p = \Gamma(\Omega,R^p\pi_* (\mathcal{K}_M\otimes\OO(|Z|-Z))$ if $p\geq 1$,
and $h_0$ is an isomorphism.
\end{thm}

Note that singularities in the boundary $b\Omega$ of $\Omega$ are permitted.
Any regular metric on $M$ will do the job if there are no singularities in the boundary of $\Omega$.
If $\Omega=X$ is compact, then the case $p=n$ can be included (see Theorem \ref{thm:mainc}).

The proof of Theorem \ref{thm:main19a} requires a detailed comparison of $\dq_{max}$- with $L^2_{loc}$-cohomology 
and of $\dq_{min}$-cohomology with cohomology with compact support (see Section \ref{sec:dolbeault}).
Theorem \ref{thm:main19a} appears below as Theorem \ref{thm:main19}.

\bigskip
Besides the $L^2$-Serre duality between the $\dq_{min}$- and the $\dq_{max}$-Dolbeault cohomology,
there is (for $p\geq 1$) another duality between the higher direct image sheaves
$R^p\pi_* \mathcal{K}_M\otimes \OO(|Z|-Z)$
on one hand and the flabby cohomology $H^{n-p}_E$ of $\OO(Z-|Z|)$ with support on the exceptional set $E=|Z|$
on the other hand. This is explained in Section \ref{subsec:HE}.
Combination of these two kinds of duality with Theorem \ref{thm:main19a}
will lead to our second main result:

\begin{thm}\label{thm:main2}
Under the assumptions of Theorem \ref{thm:main19a}, let $0 \leq  q \leq n$ if $\Omega=X$ is compact and $0 < q \leq n$ otherwise.
Then there exists a natural exact sequence
\begin{eqnarray*}
0 \rightarrow H^q_E(\wt{\Omega},\OO(Z-|Z|)) \rightarrow H^{0,q}_{max}(\wt{\Omega},L_{Z-|Z|}) \rightarrow H^{0,q}_{max}(\Omega^*) \rightarrow 0,
\end{eqnarray*}
where $H_E^*$ is the flabby cohomology with support on the exceptional set $E=|Z|$.
In case $q=n$, $H^n_E(\wt{\Omega},\OO(Z-|Z|))$ has to replaced by $0$.
\end{thm}

Note again that singularities in the boundary $b\Omega$ of $\Omega$ are permitted
and that any regular metric on $M$ will do the job if there are no singularities in the boundary of $\Omega$.
Note also that $H^0_E(\wt{\Omega},\OO(Z-|Z|))=0$ by the identity theorem.

\bigskip
The idea to identify the kernel of the natural map
$$H^{0,q}_{max}(\wt{\Omega},L_{Z-|Z|}) \rightarrow H^{0,q}_{max}(\Omega^*)$$
as the flabby cohomology of $\OO(Z-|Z|)$ with support on $E$
is inspired by the work of {\O}vrelid and Vassiliadou \cite{OvVa4} who proved Theorem \ref{thm:main2}
recently in the cases $q=n-1$ and $q=n$ (see \cite{OvVa4}, Theorem 1.4 and Corollary 1.6).
After the present paper appeared as a preprint, {\O}vrelid and Vassiliadou
added another proof of our Theorem \ref{thm:main2} for $0<q<n-1$ also to their paper (see \cite{OvVa4}, Remark 4.5.1).
Their method is quite different from our approach via Theorem \ref{thm:main1}.\footnote{
We may remark that {\O}vrelid and Vassiliadou also use some results from \cite{eins}.}

For $0<q<n-1$, one can show that
\begin{eqnarray*}
\frac{H^{0,q}_{max}(\wt{\Omega},L_{Z-|Z|})}{H^q_E(\wt{\Omega},\OO(Z-|Z|))}
&\cong& H^{0,q}_{max}\big(\wt{\Omega}\big),
\end{eqnarray*}
so that we can recover by use of Theorem \ref{thm:main2} the following, somewhat nicer result of
{\O}vrelid and Vassiliadou (\cite{OvVa4}, Theorem 1.3):
\begin{eqnarray*}
H^{0,q}_{max}\big(\wt{\Omega}\big) &\cong& H^{0,q}_{max}(\Omega^*)\ , \ \ 0<q<n-1.
\end{eqnarray*}

\medskip
If $X$ is a projective surface with only isolated singularities and $\pi: M\rightarrow X$ a resolution of singularities
with only normal crossings, then {\O}vrelid and Vassiliadou showed that $H^1_E(M,\OO(Z-|Z|))=0$
(see \cite{OvVa4}, Corollary 5.2).

Combining this with Theorem \ref{thm:main2} and Theorem A from \cite{PS1}
or Theorem 1.5 from \cite{eins}, we obtain the following result for a projective surface
considered as a Hermitian space with the restriction of the Fubini-Study metric:

\begin{thm}\label{thm:main3}
Let $X$ be a projetive surface with only isolated singularities
and $\pi: M\rightarrow X$ a resolution of singularities with only normal crossings.
Then there exist for all $0\leq q\leq 2$ natural isomorphisms
$$H^{2,q}_{max}(X-\Sing X) \rightarrow H^{2,q}(M).$$
Let $Z:=\pi^{-1}(\Sing X)$ be the unreduced exceptional divisor.
Then there exist for all $0\leq q\leq 2$ natural isomorphisms
$$H^{0,q}(M,L_{Z-|Z|}) \rightarrow H^{0,q}_{max}(X-\Sing X),$$
where $L_{Z-|Z|}\rightarrow M$ is a Hermitian holomorphic line bundle such that holomorphic sections in $L_{Z-|Z}$
correspond to sections in $\OO(Z-|Z|)$.
\end{thm}

This gives an almost complete description of the $L^2$-cohomology
of a projective surface with isolated singularities. Only the middle cohomology $H^{1,1}$ is missing.
Theorem \ref{thm:main3} was conjectured and mostly proven by Pardon and Stern in \cite{PS1}
(there was a difficulty with the critical group $H^{0,1}_{max}$).
This difficulty has been first overcome completely by {\O}vrelid and Vassiliadou in \cite{OvVa4}.

\bigskip

The organization of the present paper is as follows. In the Sections \ref{sec:cohomology}, \ref{sec:review} and \ref{sec:KXs},
we provide the necessary tools for the proof of Theorem \ref{thm:main1} in Section \ref{sec:main1}.
More precisely, in Section \ref{sec:cohomology} we show how to deal with the cohomology of non-exact complexes as e.g.
the direct image complex $\pi_* \mathcal{G}^{n,*}$
mentioned above. Section \ref{sec:review} contains a review of the $\dq_s$-complex as it is introduced in \cite{eins}.
This comprises the definition of the canonical sheaf of holomorphic $n$-forms with Dirichlet boundary condition $\mathcal{K}_X^s$
(see \eqref{eq:KXs}) and the exactness of the $\dq_s$-complex \eqref{eq:intro1} (cf. Theorem \ref{thm:exactness2}).
In Section \ref{sec:KXs}, we recall from \cite{eins} how $\mathcal{K}_X^s$ can be represented as 
the direct image of an invertible sheaf under a resolution of singularities (see Theorem \ref{thm:Ks}).
Using all these preliminaries, we prove Theorem \ref{thm:main1} in Section \ref{sec:main1}.

In Section \ref{sec:dolbeault} we study how the concepts that appear in the context of Theorem \ref{thm:main1}
are related to $L^2$-Dolbeault cohomology and prove Theorem \ref{thm:main19a}. Section \ref{sec:main2}
finally contains the proof of our second main result, Theorem \ref{thm:main2}.

\bigskip

An outline of the historical development of the topic can be found in the introduction
of the previous paper \cite{eins} and in \cite{OvVa4}.

\bigskip
{\bf Acknowledgements.} The author thanks Nils {\O}vrelid for many interesting and very helpful discussions on the topic.
This research was supported by the Deutsche Forschungsgemeinschaft (DFG, German Research Foundation), 
grant RU 1474/2 within DFG's Emmy Noether Programme.

\bigskip

\section{Cohomology of non-exact complexes}\label{sec:cohomology}

Let $X$ be a paracompact Hausdorff space.
In this section, we consider a complex of sheaves of abelian groups
\begin{eqnarray}\label{eq:complexA1}
0\rightarrow \mathcal{A} \hookrightarrow \mathcal{A}^0 \overset{a_0}{\longrightarrow} \mathcal{A}^1 
\overset{a_1}{\longrightarrow} \mathcal{A}^2 \overset{a_2}{\longrightarrow} \mathcal{A}^3 \longrightarrow  ...
\end{eqnarray}
over $X$
which is exact at $\mathcal{A}$ and $\mathcal{A}^0$ such that $\mathcal{A}\cong\ker a_0$, but not necessarily exact at $\mathcal{A}^p$, $p\geq 1$.
We denote by $\mathcal{K}^p=\ker a_p$ the kernel of $a_p$ and by $\mathcal{I}^p=\Im a_p$ the image of $a_p$.
Note that $\mathcal{K}^0\cong \mathcal{A}$. We will now represent the $q$-th (flabby) cohomology group $H^q(X,\mathcal{A})$ of $\mathcal{A}$ over $X$
by use of the $q$-th cohomology group of the complex \eqref{eq:complexA1}.
Since \eqref{eq:complexA1} is not a resolution of $\mathcal{A}$,
this also involves the quotient sheaves
\begin{eqnarray}\label{eq:complexA2}
\mathcal{R}^p := \mathcal{K}^p / \mathcal{I}^{p-1}\ , \ \ p\geq 1,
\end{eqnarray}
which are well-defined since $a_p\circ a_{p-1}=0$.
To compute the (flabby) cohomology of $\mathcal{A}$, we first require that the sheaves $\mathcal{A}^p$
are acyclic:

\begin{lem}\label{lem:complexA1}
Assume that the sheaves $\mathcal{A}^p$, $p\geq 0$, in the complex \eqref{eq:complexA1}
are acyclic. Then there are natural isomorphisms
\begin{eqnarray}\label{eq:complexA3}
\delta^0: \frac{\Gamma(X,\mathcal{I}^p)}{a_p(\Gamma(X,\mathcal{A}^p))} &\overset{\cong}{\longrightarrow}& H^1(X,\mathcal{K}^p)\ \ ,\\
\delta^q: H^q(X,\mathcal{I}^p) &\overset{\cong}{\longrightarrow}& H^{q+1}(X,\mathcal{K}^p)\label{eq:complexA4}
\end{eqnarray}
for all $p\geq 0$, $q\geq 1$.
\end{lem}

In this lemma and throughout the whole section, we may as well consider global sections $\Gamma_{cpt}$ 
and cohomology $H_{cpt}$ with compact support.

\begin{proof}
We shortly repeat the proof which is standard.
For any $p\geq 0$, we consider the short exact sequence
\begin{eqnarray}\label{eq:complexA5}
0\rightarrow \mathcal{K}^p \hookrightarrow \mathcal{A}^p \overset{a_p}{\longrightarrow} \mathcal{I}^p\rightarrow 0.
\end{eqnarray}
Using $H^r(X,\mathcal{A}^p)=0$ for $r\geq 1$, we obtain the exact cohomology sequences
\begin{eqnarray*}
&0 \rightarrow \Gamma(X,\mathcal{K}^p) \hookrightarrow \Gamma(X,\mathcal{A}^p) \overset{a_p}{\longrightarrow}
 \Gamma(X,\mathcal{I}^p) \overset{\delta^0}{\longrightarrow} H^1(X,\mathcal{K}^p) \rightarrow 0\ ,&\\
& 0 \rightarrow H^q(X,\mathcal{I}^p) \overset{\delta^q}{\longrightarrow} H^{q+1}(X,\mathcal{K}^p) \rightarrow 0\ ,\ q\geq 1,&
\end{eqnarray*}
which yield the statement of the lemma.

The assertion that the isomorphisms are natural means the following.
Let 
\begin{eqnarray}\label{eq:complexA11}
0\rightarrow \mathcal{B} \hookrightarrow \mathcal{B}^0 \overset{b_0}{\longrightarrow} \mathcal{B}^1 
\overset{b_1}{\longrightarrow} \mathcal{B}^2 \overset{b_2}{\longrightarrow} \mathcal{B}^3 \longrightarrow  ...
\end{eqnarray}
be another such complex and
$$
\begin{xy}
  \xymatrix{
      0 \ar[r] & \mathcal{A} \ar[r] \ar[d]^f    &   \mathcal{A}^* \ar[d]^g  \\
      0 \ar[r] & \mathcal{B} \ar[r]             &   \mathcal{B}^*   
  }
\end{xy}$$
a morphism of complexes (where we abbreviate the complex \eqref{eq:complexA1} by $0\rightarrow \mathcal{A} \rightarrow \mathcal{A}^*$
and the complex \eqref{eq:complexA11} by $0\rightarrow \mathcal{B} \rightarrow \mathcal{B}^*$). 
If we denote by $\mathcal{L}^p=\ker b_p$ the kernel of $b_p$ and by $\mathcal{J}^p=\Im b_p$ the image of $b_p$,
then
$$
\begin{xy}
  \xymatrix{
    \frac{\Gamma(X,\mathcal{I}^p)}{a_p(\Gamma(X,\mathcal{A}^p))}   \ar[r]^{\delta_0} \ar[d]^{g'_{p+1}}    &   H^1(X,\mathcal{K}^p) \ar[d]^{g''_{p}}  \\
    \frac{\Gamma(X,\mathcal{J}^p)}{b_p(\Gamma(X,\mathcal{B}^p))}  \ar[r]^{\delta_0}             &   H^1(X,\mathcal{L}^p)  
  }
\end{xy}$$
is commutative, where $g'_{p+1}$ and $g''_{p}$ are the maps induced by the commutative diagram
$$
\begin{xy}
  \xymatrix{
      0 \ar[r] & \mathcal{K}^p \ar[r] \ar[d]^{g_p}    &   \mathcal{A}^p \ar[r]^{a_p} \ar[d]^{g_p} & \mathcal{I}^p \ar[r] \ar[d]^{g_{p+1}}& 0  \\
      0 \ar[r] & \mathcal{L}^p \ar[r]             &   \mathcal{B}^p \ar[r]^{b_p}  & \mathcal{J}^p \ar[r] & 0.
  }
\end{xy}$$
The statement that the isomorphisms $\delta^q$ are natural follows analogously.
\end{proof}

To go on, we need some assumptions on $\mathcal{R}^p$:

\begin{lem}\label{lem:complexA2}
For $p\geq 1$, assume that the quotient sheaf $\mathcal{R}^{p}$ of the complex \eqref{eq:complexA1}
defined in \eqref{eq:complexA2} is acyclic and that the natural mapping $\Gamma(X,\mathcal{K}^p) \rightarrow \Gamma(X,\mathcal{R}^p)$
is surjective. Then there is a natural isomorphism
\begin{eqnarray*}
H^q(X,\mathcal{I}^{p-1}) \overset{\cong}{\longrightarrow} H^q(X,\mathcal{K}^{p})
\end{eqnarray*}
for all $q\geq 1$ (induced by the inclusion $\mathcal{I}^{p-1} \hookrightarrow \mathcal{K}^{p}$).
\end{lem}

\begin{proof}
Under the assumptions,
the proof follows directly from the long exact cohomology sequence that is obtained from the short
exact sequence
\begin{eqnarray}\label{eq:complexR1}
0 \rightarrow \mathcal{I}^{p-1} \hookrightarrow \mathcal{K}^{p} \longrightarrow \mathcal{R}^{p} = \mathcal{K}^p / \mathcal{I}^{p-1} \rightarrow 0.
\end{eqnarray}
If $0\rightarrow \mathcal{B} \rightarrow \mathcal{B}^*$ is another such complex,
then we obtain commutative diagrams as in the proof of Lemma \ref{lem:complexA1} showing
that the isomorphism is natural.
\end{proof}

Note that the assumptions of Lemma \ref{lem:complexA2} are trivially fulfilled if the complex \eqref{eq:complexA1}
is exact such that $\mathcal{R}^p = 0$ for all $p\geq 1$. 
From Lemma \ref{lem:complexA1} and Lemma \ref{lem:complexA2}, we deduce by induction:

\begin{lem}\label{lem:complexA3}
Under the assumptions of Lemma \ref{lem:complexA1} and Lemma \ref{lem:complexA2},
there are
natural isomorphisms
\begin{eqnarray*}
\gamma^p: \frac{\Gamma(X,\mathcal{I}^{p-1})}{a_{p-1}(\Gamma(X,\mathcal{A}^{p-1}))} &\overset{\cong}{\longrightarrow}& H^p(X,\mathcal{A})
\end{eqnarray*}
for all $p\geq 1$. Here, natural means the following. If $0\rightarrow \mathcal{B}\rightarrow \mathcal{B}^*$
is another such complex as in \eqref{eq:complexA11} and
$$\begin{xy}
  \xymatrix{
      0 \ar[r] & \mathcal{A} \ar[r] \ar[d]^f    &   \mathcal{A}^* \ar[d]^g  \\
      0 \ar[r] & \mathcal{B} \ar[r]             &   \mathcal{B}^*   
  }
\end{xy}$$
a morphism of complexes as in the proof of Lemma \ref{lem:complexA1}, then we obtain a commutative diagram
$$
\begin{xy}
  \xymatrix{
    \frac{\Gamma(X,\mathcal{I}^{p-1})}{a_{p-1}(\Gamma(X,\mathcal{A}^{p-1}))}   \ar[r]^{\gamma^p} \ar[d]^{g'_{p}}    &   H^p(X,\mathcal{A}) \ar[d]^{f_{p}}  \\
    \frac{\Gamma(X,\mathcal{J}^{p-1})}{b_{p-1}(\Gamma(X,\mathcal{B}^{p-1}))}  \ar[r]^{\gamma^p}             &   H^p(X,\mathcal{B})  ,
  }
\end{xy}$$
where $\mathcal{J}^{p-1} = \Im b_{p-1}$, $f_p$ is the map on cohomology induced by $f: \mathcal{A}\rightarrow \mathcal{B}$
and $g'_{p}$ is the map induced by $g_p: \mathcal{A}^p \rightarrow \mathcal{B}^p$.
If $f$ is an isomorphism, then $g'_{p}$ is an isomorphism as well.
\end{lem}

We can now make the connection to the $p$-th cohomology group of the complex \eqref{eq:complexA1}
by use of:

\begin{lem}\label{lem:complexA4}
For $p\geq 1$, assume that the natural mapping $\Gamma(X,\mathcal{K}^p)\rightarrow \Gamma(X,\mathcal{R}^p)$ 
is surjective.
Then there is a natural injective homomorphism
\begin{eqnarray*}
i^p: \frac{\Gamma(X,\mathcal{I}^{p-1})}{a_{p-1}(\Gamma(X,\mathcal{A}^{p-1}))}
\longrightarrow
H^p(\Gamma(X,\mathcal{A}^*)) = \frac{\Gamma(X,\mathcal{K}^p)}{a_{p-1}(\Gamma(X,\mathcal{A}^{p-1}))}
\end{eqnarray*}
with $\coker i^p = \Gamma(X,\mathcal{R}^p)$. 
More precisely, the natural sequence
\begin{eqnarray*}
0\rightarrow \frac{\Gamma(X,\mathcal{I}^{p-1})}{a_{p-1}(\Gamma(X,\mathcal{A}^{p-1}))}
\overset{i^p}{\longrightarrow}
H^p(\Gamma(X,\mathcal{A}^*)) \longrightarrow \Gamma(X,\mathcal{R}^p) \rightarrow 0
\end{eqnarray*}
is exact.
\end{lem}

\begin{proof}
From \eqref{eq:complexR1} we obtain by use of the assumption the exact sequence
$$0\rightarrow \Gamma(X,\mathcal{I}^{p-1}) \longrightarrow \Gamma(X,\mathcal{K}^p) \longrightarrow \Gamma(X,\mathcal{R}^p) \rightarrow 0,$$
and this induces the natural exact sequence
$$0\rightarrow \frac{\Gamma(X,\mathcal{I}^{p-1})}{a_{p-1}(\Gamma(X,\mathcal{A}^{p-1}))} 
\overset{i^p}{\longrightarrow}
\frac{\Gamma(X,\mathcal{K}^p)}{a_{p-1}(\Gamma(X,\mathcal{A}^{p-1}))}
\longrightarrow \Gamma(X,\mathcal{R}^p) \rightarrow 0$$
since
$$a_{p-1}(\Gamma(X,\mathcal{A}^{p-1})) \subset \Gamma(X,\mathcal{I}^{p-1}) \subset \Gamma(X,\mathcal{K}^p).$$
\end{proof}

Combining Lemma \ref{lem:complexA3} and Lemma \ref{lem:complexA4}, we conclude finally:

\begin{thm}\label{thm:complexA1}
Under the assumptions of Lemma \ref{lem:complexA1} and Lemma \ref{lem:complexA2},
there is for all $p\geq 1$ a natural injective homomorphism
\begin{eqnarray*}
i^p\circ (\gamma^p)^{-1}: H^p(X,\mathcal{A}) \longrightarrow H^p(\Gamma(X,\mathcal{A}^*))
\end{eqnarray*}
with $\coker i^p\circ (\gamma^p)^{-1} = \Gamma(X,\mathcal{R}^p)$. 
More precisely, there is a natural exact sequence
\begin{eqnarray*}
0\rightarrow H^p(X,\mathcal{A}) \overset{i^p\circ (\gamma^p)^{-1}}{\longrightarrow} H^p(\Gamma(X,\mathcal{A}^*))
\longrightarrow \Gamma(X,\mathcal{R}^p) \rightarrow 0.
\end{eqnarray*}
Here, natural means the following. If $0\rightarrow \mathcal{B}\rightarrow \mathcal{B}^*$
is another such complex as in \eqref{eq:complexA11} and
$$
\begin{xy}
  \xymatrix{
      0 \ar[r] & \mathcal{A} \ar[r] \ar[d]^f    &   \mathcal{A}^* \ar[d]^g  \\
      0 \ar[r] & \mathcal{B} \ar[r]             &   \mathcal{B}^*   
  }
\end{xy}$$
a morphism of complexes, then we obtain the commutative diagram
\begin{eqnarray*}
\begin{xy}
  \xymatrix{
    \Gamma(X,\mathcal{R}^p) \ar[d]^{g_p'''} & H^p(\Gamma(X,\mathcal{A}^*)) \ar[l] \ar[d]^{[g_p]}& 
\frac{\Gamma(X,\mathcal{I}^{p-1})}{a_{p-1}(\Gamma(X,\mathcal{A}^{p-1}))} \ar[l]_{i^p}  \ar[r]^{\gamma^p} \ar[d]^{g'_{p}}    &   H^p(X,\mathcal{A}) \ar[d]^{f_{p}}  \\
    \Gamma(X,\mathcal{S}^p) & H^p(\Gamma(X, \mathcal{B}^*))  \ar[l] &
\frac{\Gamma(X,\mathcal{J}^{p-1})}{b_{p-1}(\Gamma(X,\mathcal{B}^{p-1}))} \ar[l]_{i^p} \ar[r]^{\gamma^p}             &   H^p(X,\mathcal{B})  ,
  }
\end{xy}
\end{eqnarray*}
where $f_p$ is the map on cohomology induced by $f: \mathcal{A}\rightarrow \mathcal{B}$
and $[g_p]$, $g'_{p}$, $g_p'''$ are the maps induced by $g_p: \mathcal{A}^p \rightarrow \mathcal{B}^p$.
The quotient sheaves $\mathcal{S}^p$ are defined for the complex $0\rightarrow \mathcal{B}\rightarrow \mathcal{B}^*$
analogously to the sheaves $\mathcal{R}^p$ for the complex $0\rightarrow \mathcal{A}\rightarrow \mathcal{A}^*$.
The maps $\gamma^p$ are bijective and the maps $i^p$ are injective.
\end{thm}

In the present paper, we need the following consequence of Theorem \ref{thm:complexA1}.
Here, we make use of our general assumption that $X$ is a paracompact Hausdorff space,
because this implies that fine sheaves are acyclic.

\begin{thm}\label{thm:complexA2}
Let $X$, $M$ be paracompact Hausdorff spaces and $\pi: M \rightarrow X$ a continuous map.
Let $\mathcal{C}$ be a sheaf (of abelian groups) over $M$ and 
\begin{eqnarray}\label{eq:complexC1}
0\rightarrow \mathcal{C} \hookrightarrow \mathcal{C}^0 \overset{c_0}{\longrightarrow} \mathcal{C}^1 \overset{c_1}{\longrightarrow} \mathcal{C}^2
\overset{c_2}{\longrightarrow} \mathcal{C}^3 \longrightarrow...
\end{eqnarray}
a fine resolution. Let $\mathcal{A} \cong \pi_* \mathcal{C}$ be a sheaf on $X$,
isomorphic to the direct image of $\mathcal{C}$,
and $0\rightarrow \mathcal{A}\rightarrow \mathcal{A}^*$ a fine resolution of $\mathcal{A}$ over $X$.

Let $\mathcal{B} := \pi_* \mathcal{C}$ be the direct image of $\mathcal{C}$
and $\mathcal{B}^* = \pi_* \mathcal{C}^*$ the direct image complex
(which is again fine but not necessarily exact).
Since \eqref{eq:complexC1} is a fine resolution,
the non-exactness of $0\rightarrow \mathcal{B} \rightarrow \mathcal{B}^*$
is measured as above by the higher direct image sheaves $\mathcal{S}^p:=R^p \pi_* \mathcal{C}$, $p\geq 1$.
Let 
$$
\begin{xy}
  \xymatrix{
      0 \ar[r] & \mathcal{A} \ar[r] \ar[d]_{\cong}^{f}    &   \mathcal{A}^* \ar[d]^g  \\
      0 \ar[r] & \mathcal{B} \ar[r]             &   \mathcal{B}^*   
  }
\end{xy}$$
be a morphism of complexes,
and assume that the complex $0\rightarrow \mathcal{B} \rightarrow \mathcal{B}^*$ satisfies the assumption
of Lemma \ref{lem:complexA2},
i.e. that the direct image sheaves $\mathcal{S}^p$ are acyclic and that the maps
$\Gamma(X,\ker b_p) \rightarrow \Gamma(X,\mathcal{S}^p)$
are surjective for all $p\geq 1$. 

Then $g$ induces for all $p\geq 1$ a natural injective homomorphism
\begin{eqnarray*}
H^p(\Gamma(X,\mathcal{A}^*)) \overset{[g_p]}{\longrightarrow} H^p(\Gamma(X,\mathcal{B}^*))
\end{eqnarray*}
with $\coker [g_p] = \Gamma(X,\mathcal{S}^p)$.

More precisely, there is a natural exact sequence
\begin{eqnarray*}
0 \rightarrow H^p(\Gamma(X,\mathcal{A}^*)) \overset{[g_p]}{\longrightarrow} H^p(\Gamma(X,\mathcal{B}^*)) 
\longrightarrow \Gamma(X,\mathcal{S}^p) \rightarrow 0.
\end{eqnarray*}
In this sequence, one can replace $H^p(\Gamma(X,\mathcal{B}^*))$ by $H^p(\Gamma(M,\mathcal{C}^*))$
because
\begin{eqnarray*}
\Gamma(X,\mathcal{B}^q) &=& \Gamma(\pi^{-1}(X),\mathcal{C}^q) = \Gamma(M,\mathcal{C}^q),\\
b_q(\Gamma(X,\mathcal{B}^q)) &=& c_q(\Gamma(M,\mathcal{C}^q))
\end{eqnarray*}
by definition for all $q\geq 0$.
\end{thm}

\begin{proof}
The proof follows directly from Theorem \ref{thm:complexA1}
which we apply to the morphism of complexes 
$$(f,g): (\mathcal{A},\mathcal{A}^*) \rightarrow (\mathcal{B},\mathcal{B}^*).$$
Note that the direct image sheaves $\mathcal{B}^q=\pi_* \mathcal{C}^q$, $q\geq 0$,
are still fine sheaves for one can push forward a partition of unity under the continuous map $\pi$.

Consider the big commutative diagram in Theorem \ref{thm:complexA1}.
Since $0\rightarrow \mathcal{A}\rightarrow \mathcal{A}^*$ is a fine resolution,
the quotient sheaves $\mathcal{R}^p$, $p\geq 1$, do vanish such that the map
$i^p$ in the upper line is an isomorphism.
By assumption, the induced map on cohomology $f_p$ is an isomorphism,
and so $g_p'$ must also be isomorphic (for the maps $\gamma^p$ are isomorphisms, as well).
But then
\begin{eqnarray*}
0 \rightarrow H^p(\Gamma(X,\mathcal{A}^*)) \overset{[g_p]}{\longrightarrow} H^p(\Gamma(X,\mathcal{B}^*)) 
\longrightarrow \Gamma(X,\mathcal{S}^p) \rightarrow 0
\end{eqnarray*}
is an exact sequence.
\end{proof}


\bigskip
\section{Review of the $\dq_s$-complex}\label{sec:review}

\subsection{Two $\dq$-complexes on singular spaces}

Let us recall some of the essential constructions from \cite{eins}.
Let $X$ be always a (singular) Hermitian complex space of pure dimension $n$ and $U\subset X$ an open subset.
On a singular space, it is most fruitful
to consider forms that are square-integrable up to the singular set.
Hence, we will use the following concept of locally square-integrable forms:
\begin{eqnarray*}
L_{loc}^{p,q}(U):=\{f \in L_{loc}^{p,q}(U-\Sing X): f|_K \in L^{p,q}(K-\Sing X)\ \forall K\subset\subset U\}.
\end{eqnarray*}
It is easy to check that the presheaves given as
$$\mathcal{L}^{p,q}(U) := L_{loc}^{p,q}(U)$$
are already sheaves $\mathcal{L}^{p,q}\rightarrow X$. On $L_{loc}^{p,q}(U)$, we denote by
$$\dq_w(U): L_{loc}^{p,q}(U) \rightarrow L_{loc}^{p,q+1}(U)$$
the $\dq$-operator in the sense of distributions on $U-\Sing X$ which is closed and densely defined.
When there is no danger of confusion, we will simply write $\dq_w$ for $\dq_w(U)$.
The subscript refers to $\dq_w$ as an operator in a weak sense.
Since $\dq_w$ is a local operator, i.e.
$$\dq_w(U)|_V = \dq_w(V)$$
for open sets $V\subset U$,
we can define the presheaves of germs of forms in the domain of $\dq_w$,
$$\mathcal{C}^{p,q}:=\mathcal{L}^{p,q}\cap \dq_w^{-1}\mathcal{L}^{p,q+1},$$
given by
$$\mathcal{C}^{p,q}(U) = \mathcal{L}^{p,q}(U) \cap\Dom\dq_w(U).$$
These are actually already sheaves
because the following is also clear: If $U=\bigcup U_\mu$ is a union of open sets, $f_\mu=f|_{U_\mu}$ and
$$f_\mu \in \Dom \dq_w(U_\mu),$$
then
$$f\in \Dom \dq_w(U)\ \ \  \mbox{ and }\ \ \  \big(\dq_w(U) f\big)|_{U_\mu} = \dq_w(U_\mu) f_\mu.$$
Moreover, it is easy to see that the sheaves $\mathcal{C}^{p,q}$ admit partitions of unity,
and so we obtain fine sequences
\begin{eqnarray}\label{eq:Cseq1}
\mathcal{C}^{p,0} \overset{\dq_w}{\longrightarrow} \mathcal{C}^{p,1} \overset{\dq_w}{\longrightarrow} \mathcal{C}^{p,2} \overset{\dq_w}{\longrightarrow} ...
\end{eqnarray}
We will see later, when we deal with resolution of singularities, that
\begin{eqnarray*}
\mathcal{K}_X := \ker \dq_w \subset \mathcal{C}^{n,0}
\end{eqnarray*}
is just the canonical sheaf of Grauert and Riemenschneider since the $L^2$-property of $(n,0)$-forms
remains invariant under modifications of the metric.

The $L^{2,loc}$-Dolbeault cohomology with respect to the $\dq_w$-operator on an open set $U\subset X$
is the cohomology of the complex \eqref{eq:Cseq1} which is denoted by $H^q(\Gamma(U,\mathcal{C}^{p,*}))$.

\bigskip
Secondly, we introduce now a suitable local realization of a minimal version of the $\dq$-operator.
This is the $\dq$-operator with a Dirichlet boundary condition at the singular set $\Sing X$ of $X$.
Let
$$\dq_s(U): L_{loc}^{p,q}(U) \rightarrow L_{loc}^{p,q+1}(U)$$
be defined as follows.\footnote{Again, we write simply $\dq_s$ for $\dq_s(U)$ 
if there is no danger of confusion.} We say that $f\in\Dom\dq_w$ is in the domain of $\dq_s$ if there exists a
sequence of forms $\{f_j\}_j \subset \Dom\dq_w \subset L_{loc}^{p,q}(U)$ with essential support away from the singular set,
$$\supp f_j \cap \Sing X = \emptyset,$$
such that
\begin{eqnarray}\label{eq:ds1}
f_j \rightarrow f &\mbox{ in }& L^{p,q}(K-\Sing X),\\
\dq_w f_j \rightarrow \dq_w f &\mbox{ in }& L^{p,q+1}(K-\Sing X)\label{eq:ds2}
\end{eqnarray}
for each compact subset $K\subset\subset U$. The subscript refers to $\dq_s$ as an extension in a strong sense.
Note that we can assume without loss of generality
(by use of cut-off functions and smoothing with Dirac sequences) that the forms $f_j$ are smooth with compact support in $U-\Sing X$.

It is now clear that
$$\dq_s(U)|_V = \dq_s(V)$$
for open sets $V\subset U$, and we can define the presheaves of germs of forms in the domain of $\dq_s$,
$$\mathcal{F}^{p,q}:=\mathcal{L}^{p,q}\cap \dq_s^{-1}\mathcal{L}^{p,q+1},$$
given by 
$$\mathcal{F}^{p,q}(U) = \mathcal{L}^{p,q}(U) \cap \Dom \dq_s(U).$$
Here, we shall check a bit more carefully that these are already sheaves:
Let $U=\bigcup U_\mu$ 
be a union of open sets, $f\in L_{loc}^{p,q}(U)$ and $f_\mu=f|_{U_\mu} \in \Dom\dq_s(U_\mu)$
for all $\mu$. We claim that $f\in\Dom\dq_s(U)$.
To see this, we can assume (by taking a refinement if necessary)
that the open cover $\mathcal{U}:=\{U_\mu\}$ is locally finite,
and choose a partition of unity $\{\varphi_\mu\}$ for $\mathcal{U}$.
On $U_\mu$ choose a sequence $\{f^\mu_j\}\subset L_{loc}^{p,q}(U_\mu)$ as in \eqref{eq:ds1}, \eqref{eq:ds2},
and consider
$$f_j := \sum_{\mu} \varphi_\mu f_j^\mu.$$
It is clear that $\{f_j\}\subset L_{loc}^{p,q}(U)$. If $K\subset\subset U$ is compact, then $K\cap \supp \varphi_\mu$
is a compact subset of $U_\mu$ for each $\mu$, so that $\{f_j^\mu\}$ and $\{\dq f_j^\mu\}$ converge in the $L^2$-sense
to $f_\mu$ resp. $\dq_w f_\mu$ on $K\cap\supp \varphi_\mu$. But then $\{f_j\}$ and $\{\dq f_j\}$ converge in the $L^2$-sense
to $f$ resp. $\dq_w f$ on $K$ (recall that the cover is locally finite) and that is what we had to show.

As for $\mathcal{C}^{p,q}$, it is clear that the sheaves $\mathcal{F}^{p,q}$ are fine,
and we obtain fine sequences
\begin{eqnarray}\label{eq:Fseq1}
\mathcal{F}^{p,0} \overset{\dq_s}{\longrightarrow} \mathcal{F}^{p,1} \overset{\dq_s}{\longrightarrow} \mathcal{F}^{p,2} \overset{\dq_s}{\longrightarrow} ...
\end{eqnarray}
We can now introduce the sheaf
\begin{eqnarray}\label{eq:KXs}
\mathcal{K}^s_X := \ker \dq_s \subset \mathcal{F}^{n,0}
\end{eqnarray}
which we may call the canonical sheaf of holomorphic $n$-forms with Dirichlet boundary condition.
One of the main objectives of the present paper is to compare different representations of the cohomology of $\mathcal{K}^s_X$.
One of them will be the $L^{2,loc}$-Dolbeault cohomology with respect to the $\dq_s$-operator on open sets $U\subset X$,
i.e. the cohomology of the complex \eqref{eq:Fseq1} which is denoted by $H^q(\Gamma(U,\mathcal{F}^{p,*}))$.

\medskip
\subsection{Local $L^2$-solvability for $(n,q)$-forms}
It is clearly interesting to study wether the sequences \eqref{eq:Cseq1} and \eqref{eq:Fseq1} are exact,
which is well-known to be the case in regular points of $X$ where the $\dq_w$- and the $\dq_s$-operator coincide.
In singular points, the situation is quite complicated for forms of arbitrary degree and not completely understood.
However, the $\dq_w$-equation is locally solvable in the $L^2$-sense at arbitrary singularities for forms
of degree $(n,q)$, $q>0$ (see \cite{PS1}, Proposition 2.1), and for forms of degree $(p,q)$, $p+q>n$, 
at isolated singularities (see \cite{FOV2}, Theorem 1.2).
Since we are concerned with canonical sheaves, we may restrict our attention to the case of $(n,q)$-forms
and conclude:

\begin{thm}\label{thm:exactness1}
Let $X$ be a Hermitian complex space of pure dimension $n$. Then
\begin{eqnarray}\label{eq:exactness1}
0\rightarrow \mathcal{K}_X \hookrightarrow \mathcal{C}^{n,0} \overset{\dq_w}{\longrightarrow}
\mathcal{C}^{n,1} \overset{\dq_w}{\longrightarrow} \mathcal{C}^{n,2} \overset{\dq_w}{\longrightarrow} ... \longrightarrow \mathcal{C}^{n,n}
\rightarrow 0
\end{eqnarray}
is a fine resolution.
For an open set $U\subset X$, it follows that
$$H^q(U,\mathcal{K}_X) \cong H^q(\Gamma(U,\mathcal{C}^{n,*}))\ , \ H^q_{cpt}(U,\mathcal{K}_X) \cong H^q(\Gamma_{cpt}(U,\mathcal{C}^{n,*})).$$
\end{thm}

Concerning the $\dq_s$-equation, local $L^2$-solvability for forms of degree $(n,q)$
is known to hold on spaces with isolated singularities (see \cite{eins}, Theorem 1.9),
but the problem is open at arbitrary singularities.

So, let $X$ have only isolated singularities. 
Then the $\dq_s$-equation is locally exact on $(n,q)$-forms
for $1\leq q \leq n$ by \cite{eins}, Lemma 5.5.
The statement was deduced from the results of Forn{\ae}ss, {\O}vrelid and Vassiliadou \cite{FOV2}.
Hence we note:

\begin{thm}\label{thm:exactness2}
Let $X$ be a Hermitian complex space of pure dimension $n\geq 2$ with only isolated singularities. Then
\begin{eqnarray}\label{eq:exactness2}
0\rightarrow \mathcal{K}_X^s \hookrightarrow \mathcal{F}^{n,0} \overset{\dq_s}{\longrightarrow}
\mathcal{F}^{n,1} \overset{\dq_s}{\longrightarrow} \mathcal{F}^{n,2} \overset{\dq_s}{\longrightarrow} ... \longrightarrow \mathcal{F}^{n,n}
\rightarrow 0
\end{eqnarray}
is a fine resolution.
For an open set $U\subset X$, it follows that
$$H^q(U,\mathcal{K}_X^s) \cong H^q(\Gamma(U,\mathcal{F}^{n,*}))\ ,\ H^q_{cpt}(U,\mathcal{K}_X^s) \cong H^q(\Gamma_{cpt}(U,\mathcal{F}^{n,*})).$$
\end{thm}

\bigskip
\section{Resolution of $(X,\mathcal{K}_X^s)$}\label{sec:KXs}


\subsection{Desingularization and comparison of metrics}\label{subsec:metrics}

Let $\pi: M \rightarrow X$
be a resolution of singularities (which exists due to Hironaka \cite{Hi}), i.e. a proper holomorphic surjection such that
\begin{eqnarray*}
\pi|_{M-E}: M-E \rightarrow X-\Sing X
\end{eqnarray*}
is biholomorphic, where $E=|\pi^{-1}(\Sing X)|$ is the exceptional set.
We may assume that $E$ is a divisor with only normal crossings,
i.e. the irreducible components of $E$ are regular and meet complex transversely.
Let $Z:=\pi^{-1}(\Sing X)$ be the unreduced exceptional divisor.
For the topic of desingularization, we refer to
\cite{AHL}, \cite{BiMi} and \cite{Ha}.
Let
$
\gamma:= \pi^* h
$
be the pullback of the Hermitian metric $h$ of $X$ to $M$.
$\gamma$ is positive semidefinite (a pseudo-metric) with degeneracy locus $E$.

We give $M$ the structure of a Hermitian manifold with a freely chosen (positive definite)
metric $\sigma$. Then $\gamma \lesssim \sigma$
and $\gamma \sim \sigma$ on compact subsets of $M-E$.
For an open set $U\subset M$, we denote by $L^{p,q}_{\gamma}(U)$ and $L^{p,q}_{\sigma}(U)$
the spaces of square-integrable $(p,q)$-forms with respect to the (pseudo-)metrics $\gamma$ and $\sigma$,
respectively. 

Since $\sigma$ is positive definite and $\gamma$ is positive semi-definite,
there exists a continuous function $g\in C^0(M,\R)$ such that
\begin{eqnarray}\label{eq:l2dV}
dV_\gamma = g^2 dV_\sigma.
\end{eqnarray}
This yields $|g| |\omega|_\gamma  = |\omega|_\sigma$
if $\omega$ is an $(n,0)$-form, and
$|\omega|_\sigma \lesssim_U |g||\omega|_\gamma$
on $U\subset\subset M$ if $\omega$ is a $(n,q)$-form, $0\leq q\leq n$.\footnote{
This statement means that $|\omega|_\sigma/|\omega|_\gamma$ is locally bounded on $M$ for $(n,q)$-forms.}
So, for an $(n,q)$ form $\omega$ on $U\subset\subset M$:
\begin{eqnarray}\label{eq:l2est2}
\int_U |\omega|_\sigma^2 dV_\sigma \lesssim_U \int_U g^{2} |\omega|_\gamma^2 g^{-2} dV_\gamma = \int_U |\omega|^2_\gamma dV_\gamma.
\end{eqnarray}
Conversely,
$|g| |\eta|_\gamma \lesssim_U |\eta|_\sigma$
on $U\subset\subset M$ if $\eta$ is a $(0,q)$-form, $0\leq q\leq n$.\footnote{
For $(0,q)$-forms, $|\omega|_\gamma/|\omega|_\sigma$ is locally bounded.}
So, for a $(0,q)$ form $\eta$ on $U\subset\subset M$:
\begin{eqnarray}\label{eq:l2est}
\int_U |\eta|_\gamma^2 dV_\gamma \lesssim_U \int_U g^{-2} |\eta|_\sigma^2 g^2 dV_\sigma = \int_U |\eta|^2_\sigma dV_\sigma.
\end{eqnarray}
For open sets $U\subset\subset M$ and all $0\leq q\leq n$, we conclude the relations
\begin{eqnarray}\label{eq:l2est3}
L^{n,q}_{\gamma}(U) &\subset& L^{n,q}_{\sigma}(U),\\
\label{eq:l2est4}
L^{0,q}_{\sigma}(U) &\subset& L^{0,q}_{\gamma}(U).
\end{eqnarray}
For an open set $\Omega \subset X$, $\Omega^*=\Omega - \Sing X$, $\wt{\Omega}:=\pi^{-1}(\Omega)$,
pullback of forms under $\pi$ gives the isometry
\begin{eqnarray}\label{eq:l2est5}
\pi^*: L^{p,q}(\Omega^*) \longrightarrow L^{p,q}_{\gamma}(\wt{\Omega}-E) \cong L^{p,q}_{\gamma}(\wt{\Omega}),
\end{eqnarray}
where the last identification is by trivial extension of forms over the thin exceptional set $E$.

\subsection{Representation of $\mathcal{K}_X^s$ under desingularization}

By use of \eqref{eq:l2est5},
both complexes, the resolution $(\mathcal{C}^{n,*},\dq_w)$ of $\mathcal{K}_X$ 
and the resolution $(\mathcal{F}^{n,*},\dq_s)$ of $\mathcal{K}_X^s$, can be studied as well
on the complex manifold $M$. This is the point of view that was taken in \cite{eins}
where we considered the sheaves $\mathcal{L}^{p,q}_\gamma$ on $M$ given by
$$\mathcal{L}^{p,q}_\gamma(U) := L^{p,q}_{\gamma,loc}(U)$$
and
\begin{eqnarray}\label{eq:C1}
\mathcal{C}^{p,q}_{\gamma,E} := \mathcal{L}^{p,q}_{\gamma} \cap \dq_{w,E}^{-1} \mathcal{L}^{p,q+1}_\gamma,
\end{eqnarray}
where $\dq_{w,E}$ is the $\dq$-operator in the sense of distributions with respect to compact subsets of $M - E$.
\eqref{eq:C1} is given by the presheaf $\mathcal{C}^{p,q}_{\gamma,E}(U) = \mathcal{L}^{p,q}_\gamma (U) \cap \Dom \dq_{w,E}(U)$.
It follows from \eqref{eq:l2est5} that $(\mathcal{C}^{p,*},\dq_w)$ can be canonically
identified with the direct image complex $(\pi_* \mathcal{C}^{p,*}_{\gamma,E},\pi_* \dq_{w,E})$.
Since $\mathcal{L}^{n,0}_\gamma = \mathcal{L}^{n,0}_\sigma$ for the regular metric $\sigma$ on $M$
by use of \eqref{eq:l2est3} and \eqref{eq:l2est4},
we can use the fact that the $\dq$-equation in the sense of distributions for $L^2_\sigma$-forms
extends over exceptional sets (see e.g. \cite{Rp1}, Lemma 2.1) to conclude that
$$\mathcal{K}_M := \ker \dq_{w,E} \subset \mathcal{L}^{n,0}_\gamma = \mathcal{L}^{n,0}_\sigma$$
is just the usual canonical sheaf on the complex manifold $M$, and that
$\mathcal{K}_X \cong \pi_* \mathcal{K}_M$
is in fact the canonical sheaf of Grauert-Riemenschneider.

Analogously, we consider now the $\dq_s$-complex.
Let $\dq_{s,E}$ be the $\dq$-operator acting on $\mathcal{L}^{p,q}_\gamma$-forms,
defined as the $\dq_s$-operator on $X$ above, but with the exceptional set $E$ in place of the singular set $\Sing X$.
Let
$$\mathcal{F}^{p,q}_{\gamma,E}:=\mathcal{L}^{p,q}_\gamma \cap \dq_{s,E}^{-1} \mathcal{L}^{p,q+1}_\gamma.$$
Then it follows from \eqref{eq:l2est5} that $(\mathcal{F}^{p,*},\dq_s)$ can be canonically
identified with the direct image complex $(\pi_* \mathcal{F}^{p,*}_{\gamma,E},\pi_* \dq_{s,E})$.

It remains to study $\ker \dq_{s,E}$ in $\mathcal{L}^{n,0}_\gamma = \mathcal{L}^{n,0}_\sigma$
because $\pi_*(\mathcal{F}^{n,0}_{\gamma,E} \cap  \ker\dq_{s,E}) = \mathcal{K}^s_X$.
This was a central point \cite{eins} (see \cite{eins}, Section 6.3):
if the resolution of singularities $\pi: M\rightarrow X$ is chosen appropriately,
then we can deduce suitable statements about $\ker\dq_{s,E}$ and obtain 
Theorem \ref{thm:Ks} from the Introduction.

To repeat it shortly, there exists a resolution of singularities $\pi: M \rightarrow X$ with only normal crossings
and an effective divisor $D\geq Z -|Z|$ with support on the exceptional set,
where $Z=\pi^{-1}(\Sing X)$ is the unreduced exceptional divisor,
such that
\begin{eqnarray}\label{eq:DD}
\mathcal{K}_X^s = \pi_* \big( \mathcal{K}_M \otimes \OO(-D) \big).
\end{eqnarray}
In many situations, e.g., if $\dim X=2$ or if $X$ has only homogeneous singularities,
then one can take $D=Z-|Z|$. In the present paper, we will assume from now on that this is actually the case,
and show how the $L^2$-theory from \cite{eins} can be improved under this assumption.

\medskip
So, for the rest of the paper, we assume that \eqref{eq:DD} holds with $D=Z-|Z|$.


\bigskip
\section{Proof of Theorem \ref{thm:main1}}\label{sec:main1}

We can use Theorem \ref{thm:complexA2} to represent the cohomology groups 
\begin{eqnarray*}
H^q(X,\mathcal{K}_X^s) &\cong& H^q(\Gamma(X,\mathcal{F}^{n,*})),\\
H^q_{cpt}(X,\mathcal{K}_X^s) &\cong& H^q(\Gamma_{cpt}(X,\mathcal{F}^{n,*}))
\end{eqnarray*}
in terms of cohomology groups on the resolution $\pi: M\rightarrow X$,
where we let $\pi: M\rightarrow X$ be a resolution of singularities as in Theorem \ref{thm:Ks} so that $D=Z-|Z|$.
Recall also that $X$ is a Hermitian complex space of pure dimension $n$
with only isolated singularities.

Clearly, we intend to use Theorem \ref{thm:complexA2} with $\mathcal{A}=\mathcal{K}_X^s$
and
$$(\mathcal{A}^*,a_*) = (\mathcal{F}^{n,*},\dq_s),$$
so that $0\rightarrow \mathcal{A} \rightarrow \mathcal{A}^*$ is a fine resolution of $\mathcal{A}$ over $X$
by Theorem \ref{thm:exactness2}.
By assumption,
$$\mathcal{A} = \mathcal{K}_X^s \cong \pi_* \big( \mathcal{K}_M \otimes \OO(|Z|-Z) \big),$$
so that we can choose
$$\mathcal{C} = \mathcal{K}_M \otimes \OO(|Z|-Z)$$
for the application of Theorem \ref{thm:complexA2}.
It remains to choose a suitable fine resolution $0\rightarrow\mathcal{C} \rightarrow \mathcal{C}^*$.
Since $(M,\sigma)$ is an ordinary Hermitian manifold,
we can use the usual $L^2_\sigma$-complex of forms with values in $\OO(|Z|-Z)$.
To realize that, we can adopt two different points of view.
First, let $L_{|Z|-Z}\rightarrow M$ be the holomorphic line bundle associated to the divisor $|Z|-Z$
such that holomorphic sections of $L_{|Z|-Z}$ correspond to sections of $\OO(|Z|-Z)$,
and give $L_{|Z|-Z}$ the structure of a Hermitian line bundle by choosing an arbitrary
positive definite Hermitian metric. Then, denote by
$$L^{p,q}_\sigma(U,L_{|Z|-Z})\ ,\ L^{p,q}_{\sigma,loc}(U,L_{|Z|-Z})$$
the spaces of (locally) square-integrable $(p,q)$-forms with values in $L_{|Z|-Z}$ (with respect to the
metric $\sigma$ on $M$ and the chosen metric on $L_{|Z|-Z}$).
We can then define the sheaves of germs of square-integrable $(p,q)$-forms
with values in $L_{|Z|-Z}$, $\mathcal{L}^{p,q}_\sigma(L_{|Z|-Z})$, by the assignment
$$\mathcal{L}^{p,q}_\sigma(L_{|Z|-Z}) (U) = L^{p,q}_{\sigma,loc}(U,L_{|Z|-Z}).$$
The second point of view is to use the sheaves $\mathcal{L}^{p,q}_\sigma \otimes \OO(|Z|-Z)$
which are canonically isomorphic to the sheaves $\mathcal{L}^{p,q}_\sigma(L_{|Z|-Z})$.
Let us keep both points of view in mind.
As in \eqref{eq:C1}, let
\begin{eqnarray}\label{eq:C2}
\mathcal{C}^{p,q}_{\sigma,E}(L_{|Z|-Z}) := \mathcal{L}^{p,q}_{\sigma}(L_{|Z|-Z}) \cap \dq_{w,E}^{-1} \mathcal{L}^{p,q+1}_\sigma(L_{|Z|-Z}),
\end{eqnarray}
where $\dq_{w,E}$ is the $\dq$-operator in the sense of distributions with respect to compact subsets of $M - E$
(and for forms with values in $L_{|Z|-Z}$). Since $\sigma$ is positive definite, the $\dq$-equation
in the sense of distributions for $L^2_\sigma$-forms with values in a holomorphic line bundle 
extends over the exceptional set, and we can drop the $E$ in the notation and use as well the $\dq_w$-operator on $M$.

It is clear that the sheaves $\mathcal{C}^{p,q}_\sigma(L_{|Z|-Z})$ are fine.
Now then, the ordinary Lemma of Dolbeault tells us that
\begin{eqnarray*}
0\rightarrow \mathcal{K}_M \otimes\OO(|Z|-Z) \hookrightarrow
\mathcal{C}^{n,0}_\sigma(L_{|Z|-Z}) \overset{\dq_w}{\longrightarrow}
\mathcal{C}^{n,1}_\sigma(L_{|Z|-Z}) \overset{\dq_w}{\longrightarrow} ...
\end{eqnarray*}
is a fine resolution, and we choose
$$(\mathcal{C}^*,c_*) = (\mathcal{C}^{n,*}_\sigma(L_{|Z|-Z}), \dq_w)$$
for the application of Theorem \ref{thm:complexA2}. Note that
\begin{eqnarray*}
H^q(M,\mathcal{K}_M\otimes \OO(|Z|-Z)) &\cong& H^q(\Gamma(M,\mathcal{C}^{n,*}_\sigma(L_{|Z|-Z}))),\\
H^q_{cpt}(M,\mathcal{K}_M\otimes\OO(|Z|-Z)) &\cong& H^q(\Gamma_{cpt}(M,\mathcal{C}^{n,*}_\sigma(L_{|Z|-Z}))).
\end{eqnarray*}

We have to consider the direct image complex
$$\mathcal{B}^* = \pi_* \mathcal{C}^* = \pi_* \mathcal{C}^{n,*}_\sigma(L_{|Z|-Z}),$$
and recall that the non-exactness of the complex 
$0\rightarrow \pi_*\mathcal{C} \rightarrow \pi_* \mathcal{C}^*$ is measured
by the higher direct image sheaves
$$\mathcal{R}^p := R^p \pi_* \big(\mathcal{K}_M \otimes\OO(|Z|-Z)\big)\ ,\ p\geq 1.$$
We have to check that these satisfy the assumptions in Theorem \ref{thm:complexA2}
(where the $\mathcal{R}^p$ will appear in place of the $\mathcal{S}^p$).
For a point $x\in X$, we have
$$(\mathcal{R}^p)_x = \lim_{\substack{\longrightarrow\\ U}} H^p\big(\pi^{-1}(U),\mathcal{K}_M\otimes\OO(|Z|-Z)\big),$$
where the limit runs over open neighborhoods of $x$ in $X$.
Since $\pi$ is a biholomorphism outside the exceptional set, it follows that
$\mathcal{R}^p$ is a skyscraper sheaf with $(\mathcal{R}^p)_x=0$ for $x\notin\Sing X$.
Hence, the sheaves $\mathcal{R}^p$ are acyclic.

It remains to check that the canonical maps $\Gamma(X,\ker b_p) \rightarrow \Gamma(X,\mathcal{R}^p)$
are surjective for $p\geq 1$. So, let $[\omega]\in\Gamma(X,\mathcal{R}^p)$. Since $\mathcal{R}^p$ is a
skyscraper sheaf as described above, $[\omega]$ is represented by a set of germs $\{\omega_x\}_{x\in\Sing X}$,
where each $\omega_x$ is given by a $\dq$-closed $(n,p)$-form with values in $L_{|Z|-Z}$ in a neighborhood
$U_x$ of the component $\pi^{-1}(\{x\})$ of the exceptional set,
$$\omega_x \in \ker\dq_w \subset \mathcal{C}^{n,p}_\sigma(L_{|Z|-Z})(U_x).$$
We will see in a moment that we can assume that the forms $\omega_x$ have compact support. So, they give rise
to a global form $\omega \in \ker\dq_w \subset \mathcal{C}^{n,p}_\sigma(L_{|Z|-Z})(M)$, i.e.
$$\omega\in  \Gamma(M,\ker\dq_w\cap \mathcal{C}^{n,p}_\sigma(L_{|Z|-Z})) = \Gamma(X,\ker b_p)$$ 
represents $[\omega]\in \Gamma(X,\mathcal{R}^p)$.

To show that we can choose $\omega_x$ with compact support in $U_x$, we can use the fact that $Z-|Z|$ is effective
so that $\omega_x$ can be interpreted as a $\dq$-closed form in $\mathcal{C}^{n,p}_\sigma(U_x)$.
But Takegoshi's vanishing theorem (see \cite{Ta}, Theorem 2.1) tells us that there is a solution $\eta_x\in \mathcal{C}^{n,p-1}_\sigma(V_x)$
to the equation $\dq_w \eta_x =\omega_x$
on a smaller neighborhood of the component $\pi^{-1}(\{x\})$ of the exceptional set.
Since $(\mathcal{C}^{n,*}_\sigma,\dq_w)$ is a fine resolution of the canonical sheaf $\mathcal{K}_M$,
this fact is also expressed by the vanishing of the higher direct image sheaves
$$R^p\pi_* \mathcal{K}_M = 0\ ,\ p\geq 1.$$
Let $\chi_x$ be a smooth cut-off function with compact support in $V_x$ that is identically $1$ in a
neighborhood of $\pi^{-1}(\{x\})$.
Then $\dq_w(\chi_x \eta_x)$ is the form we were looking for because it has compact support in $U_x$ and equals
$\omega_x$ in a neighborhood of $\pi^{-1}(\{x\})$ so that it can be considered as a form with values in $L_{|Z|-Z}$ again.

Hence, we can use Theorem \ref{thm:complexA2} with
\begin{eqnarray*}
& (\mathcal{A},\mathcal{A}^*)& = \big(\mathcal{K}_X^s, \mathcal{F}^{n,*}\big),\\
& (\mathcal{B},\mathcal{B}^*)& = \big(\pi_*(\mathcal{K}_M\otimes\OO(|Z|-Z)), \pi_* \mathcal{C}^{n,*}_\sigma(L_{|Z|-Z})\big),\\
& (\mathcal{C},\mathcal{C}^*)& = \big(\mathcal{K}_M\otimes\OO(|Z|-Z), \mathcal{C}^{n,*}_\sigma(L_{|Z|-Z})\big)
\end{eqnarray*}
after choosing a suitable morphism of complexes $(f,g): (\mathcal{A},\mathcal{A}^*) \rightarrow (\mathcal{B},\mathcal{B}^*)$ which gives
the commutative diagram
$$
\begin{xy}
  \xymatrix{
      0 \ar[r] & \mathcal{A} \ar[r] \ar[d]_{\cong}^{f}    &   \mathcal{A}^* \ar[d]^g  \\
      0 \ar[r] & \mathcal{B} \ar[r]             &   \mathcal{B}^*
  }
\end{xy}$$
It turns out that we can simply use the natural inclusion because the sheaves $\mathcal{A}^p = \mathcal{F}^{n,p}$
are actually subsheaves of $\mathcal{B}^p = \pi_* \mathcal{C}^{n,p}_\sigma(L_{|Z|-Z})$ for all $p\geq0$:

\begin{lem}\label{lem:inclusion}
For all $p\geq 0$, we have
$$\mathcal{F}^{n,p}_{\gamma,E} \subset \mathcal{C}^{n,p}_\sigma(L_{|Z|-Z})$$
as subsheaves of $\mathcal{L}^{n,p}_\sigma$. It follows that
$$\mathcal{F}^{n,p} \cong \pi_* \mathcal{F}^{n,p}_{\gamma,E} \subset \pi_* \mathcal{C}^{n,p}_\sigma(L_{|Z|-Z}).$$
\end{lem}

\begin{proof}
The proof is contained in the proof of Lemma 6.2 in \cite{eins},
but we should repeat the arguments here
as this lemma really shows why we can use the pull-back of forms to realize
the cohomology-mappings explicitly.
We observe that
$$\mathcal{F}^{n,p}_{\gamma,E} \subset \mathcal{L}^{n,p}_\gamma \subset \mathcal{L}^{n,p}_\sigma$$
by definition of $\mathcal{F}^{n,p}_{\gamma,E}$ and \eqref{eq:l2est3}.
On the other hand, there is a natural inclusion
$\mathcal{L}^{n,p}_\sigma(L_{|Z|-Z}) \subset \mathcal{L}^{n,p}_\sigma$
since $Z-|Z|$ is an effective divisor so that
$$\mathcal{C}^{n,p}_\sigma(L_{|Z|-Z}) \subset \mathcal{L}^{n,p}_\sigma(L_{|Z|-Z}) \subset \mathcal{L}^{n,p}_\sigma$$
by definition of $\mathcal{C}^{n,p}_\sigma(L_{|Z|-Z})$.

Since the statement is local, it is enough to consider a point $P\in E$ and a neighborhood $U$ of $P$ such that
$U$ is an open set in $\C^n$, that $E$ is the normal crossing $\{z_1\cdots z_d=0\}$,
and $P=0$. For $\sigma$ we can take the Euclidean metric.

Let us investigate the behavior of $(n,q)$-forms under the resolution $\pi: M\rightarrow X$
at the isolated singularity $\pi(P)$. We can assume that a neighborhood of $\pi(P)$ is embedded
holomorphically into $W\subset\subset\C^L$, $L\gg n$, such that $\pi(P)=0$,
and that $\gamma=\pi^* h$ where $h$ is the Euclidean metric in $\C^L$.
Let $w_1, ..., w_L$ be the Cartesian coordinates of $\C^L$.
We are interested in the behavior of the forms $\eta_\mu:=\pi^* d\o{w_\mu}$ at the exceptional set.
Let $dz_N:=dz_1\wedge\cdots \wedge dz_n$.
We claim that a form $\alpha$ is in $L^{n,q}_{\gamma}(U)$ exactly if it can be written
in multi-index notation as
\begin{eqnarray}\label{eq:alpha01}
\alpha = \sum_{|K|=q} \alpha_K dz_N\wedge \eta_K = dz_N \wedge \sum_{|K|=q} \alpha_K \eta_K
\end{eqnarray}
with coefficients $\alpha_K \in L^{0,0}_\sigma(U)$.
That can be seen as follows. 
Let $g$ be a function as in Subsection \ref{subsec:metrics} such that $|dz_N|_\gamma = |g|^{-1}$
and set $\Omega=g dz_N$. Note that $|\Omega|_\gamma\equiv 1$. Let $\alpha\in L^{n,q}_{\gamma}(U)$.
Then $\alpha=\Omega\wedge\wt{\alpha}$ and $|\alpha|_\gamma=|\wt{\alpha}|_\gamma$, hence $\wt{\alpha}\in L^{0,q}_\gamma(U)$.
There is a $(0,q)$-form $A$ on (the regular part of $\pi(U)$) such that $\wt{\alpha}=\pi^* A$.
Let $i: \pi(U)\hookrightarrow \C^L$ be the inclusion.
$A$ has a unique representation $A=\sum_{|K|=q} A_K i^* d\o{w_K}$ with coefficients $A_K$ such that $|A|_h^2=\sum |A_K|_h^2$.
This follows from the following

\smallskip
{\bf Proposition.} {\em Let $N$ be a complex submanifold in $\C^L$ with the restriction of the Euclidean metric,
$i: N\hookrightarrow \C^L$ the inclusion. Let $\eta$ be a $(p,q)$-form on $N$.
Then there are uniquely determined coefficients $a_{PQ}$ such that $\eta=\sum a_{PQ} i^* (dw_P\wedge d\o{w_Q})$
and $|\eta|^2=\sum |a_{PQ}|^2$, where $w_1, ..., w_L$ are the Euclidean coordinates of $\C^L$.}

\smallskip
Hence, $|\alpha|^2_\gamma=|\wt{\alpha}|^2_\gamma=|A|^2_h=\sum|A_K|^2_h=\sum |\pi^* A_K|^2_\gamma$
and all the $\pi^* A_K$ are in $L^{0,0}_\gamma(U)$. But then $\alpha_K:=g\cdot \pi^* A_K \in L^{0,0}_\sigma(U)$
and \eqref{eq:alpha01} holds with these coefficients.
The converse direction of the claim \eqref{eq:alpha01} is similar and easier to see.

Let $Z$ have the order $k_j\geq 1$ on $\{z_j=0\}$, i.e. assume that $Z$ is given by
$f=z_1^{k_1}\cdots z_d^{k_d}$.
Since $Z=\pi^{-1}(\Sing X)$,
each $\pi^* w_\mu$ must vanish of order $k_j$ on $\{z_j=0\}$.
We conclude that $\pi^* w_\mu$
has a factorization
\begin{eqnarray*}
\pi^* w_\mu = f g_\mu = z_1^{k_1} \cdots z_d^{k_d}\cdot g_\mu,
\end{eqnarray*}
where $g_\mu$ is a holomorphic function on $U$. So,
\begin{eqnarray*}
\eta_\mu = \pi^* d\o{w_\mu} = d\pi^*\o{w_\mu} = \big(\o{z_1}^{k_1-1}\cdots \o{z_d}^{k_d-1}\big) \cdot \beta_\mu,
\end{eqnarray*}
where the $\beta_\mu$ are $(0,1)$-forms that are bounded with respect to the non-singular metric $\sigma$.
This means that $\eta_\mu = \pi^*d\o{w_\mu}$ vanishes at least to the order of $Z-|Z|$ along the exceptional set $E$
(with respect to the metric $\sigma$). 

So, \eqref{eq:alpha01} implies that a form $\alpha$ is in $L^{n,q}_{\gamma}(U)$ exactly if it can be written
in multi-index notation as
\begin{eqnarray}\label{eq:alpha02}
\alpha = \big(\o{z_1}^{k_1-1}\cdots \o{z_d}^{k_d-1}\big)^q \sum_{|K|=q} \alpha_K dz_N\wedge \beta_K
\end{eqnarray}
with coefficients $\alpha_K \in L^{0,0}_\sigma(U)$. 

We conclude that
$\mathcal{F}^{n,p}_{\gamma,E} \subset \mathcal{L}^{n,p}_\gamma \subset \mathcal{L}^{n,p}_\sigma(L_{|Z|-Z})$
for all $p\geq 1$, and it remains to treat the case $p=0$.
So, let $\phi \in \mathcal{F}^{n,0}_{\gamma,E}(U)$.
This means that there exists $\psi\in \mathcal{L}^{n,1}_\gamma(U) \subset \mathcal{L}^{n,1}_\sigma(L_{|Z|-Z})(U)$
such that $\dq_{s,E} \phi=\psi$.
But it was shown in \cite{eins}, Lemma 6.2, that the left-hand side of $\dq_{s,E}\phi=\psi$
is in $\mathcal{L}^{n,0}_\sigma(L_{|Z|-Z})$ if the right-hand side is in $\mathcal{L}^{n,1}_\sigma(L_{|Z|-Z})$.
This was elaborated for $\psi\equiv 0$, but we can as well 
insert any $\psi$.
The key point is that there exists a sequence of smooth forms $\phi_j$ with support away from $E$
such that 
\begin{eqnarray*}
\phi_j \rightarrow \phi &\mbox{ in }& L^{n,0}_\gamma(V) = L^{n,0}_\sigma(V),\\
\dq \phi_j \rightarrow \psi &\mbox{ in }& L^{n,1}_\gamma(V) 
\end{eqnarray*}
on suitable open sets $V\subset U$. The considerations above show that convergence in $L^{n,1}_\gamma(V)$
implies convergence in $L^{n,1}_\sigma(V,L_{|Z|-Z})$.
By use of the inhomogeneous Cauchy formula, one can show that this implies convergence of $\{\phi_j\}_j$
in $L^{n,0}_\sigma(V,L_{|Z|-Z})$, as well (see \cite{eins}, Lemma 6.2). But then $\phi\in \mathcal{L}^{n,0}_\sigma(L_{|Z|-Z})$.

So, we have seen that
\begin{eqnarray}\label{eq:inc1}
\mathcal{F}^{n,p}_{\gamma,E} &\subset& \mathcal{L}^{n,p}_\sigma(L_{|Z|-Z})\ ,\ p\geq 0,\\
\mathcal{L}^{n,p}_\gamma &\subset& \mathcal{L}^{n,p}_\sigma(L_{|Z|-Z})\ ,\ p\geq 1.\label{eq:inc2}
\end{eqnarray}
Let $p\geq 0$ and $\phi\in \mathcal{F}^{n,p}_{\gamma,E}(U)$ on an open set $U$.
Then $\phi\in \mathcal{L}^{n,p}_\sigma(U,L_{|Z|-Z})$ by \eqref{eq:inc1} 
and $\dq_{s,E}\phi\in \mathcal{L}^{n,p+1}_\sigma(U,L_{|Z|-Z})$ by \eqref{eq:inc2}.
By definition, $\dq_{w,E}\phi=\dq_{s,E}\phi$ and thus $\phi\in \mathcal{C}^{n,p}_\sigma(L_{|Z|-Z})(U)$.
\end{proof}

So, we can finally apply Theorem \ref{thm:complexA2} and conclude that (for $p\geq 1$) the natural inclusion 
$\iota: \mathcal{F}^{n,*} \hookrightarrow \pi_* \mathcal{C}^{n,*}_\sigma(L_{Z|-Z})$ induces
the natural injective homomorphisms
\begin{eqnarray*}
H^p\big(\Gamma(X,\mathcal{F}^{n,*})\big) &\overset{[\iota_p]}{\longrightarrow}& H^p\big(\Gamma(X,\pi_*\mathcal{C}^{n,*}_\sigma(L_{|Z|-Z}))\big),\\
H^p\big(\Gamma_{cpt}(X,\mathcal{F}^{n,*})\big) &\overset{[\iota_p]}{\longrightarrow}& H^p\big(\Gamma_{cpt}(X,\pi_*\mathcal{C}^{n,*}_\sigma(L_{|Z|-Z}))\big)
\end{eqnarray*}
with $\coker [\iota_p] = \Gamma(X,\mathcal{R}^p)$ in both cases. For $p=0$, it is clear that the
maps $[\iota_0]$ 
are isomorphisms.
Since pull-back of forms under $\pi$ gives an isomorphism
$$\pi^*: \Gamma\big(X,\pi_* \mathcal{C}^{n,*}_\sigma(L_{|Z|-Z})\big) \overset{\cong}{\longrightarrow}
\Gamma\big(M,\mathcal{C}^{n,*}_\sigma(L_{|Z|-Z})\big),$$
we conclude (omitting the natural inclusions $\iota$, $[\iota_p]$ from the statement):

\begin{thm}\label{thm:main0}
Let $X$ be a Hermitian complex space of pure dimension $n\geq 2$ with only isolated singularities,
and $\pi: M\rightarrow X$ a resolution of singularities with only normal crossings
such that we can take $D=Z-|Z|$ in Theorem \ref{thm:Ks}.

Then the pull-back of forms $\pi^*$ induces for all $0\leq p\leq n$ natural injective homomorphisms
\begin{eqnarray*}
H^p\big(\Gamma(X,\mathcal{F}^{n,*})\big) &\overset{[\pi^*_p]}{\longrightarrow}& H^p\big(\Gamma(M,\mathcal{C}^{n,*}_\sigma(L_{|Z|-Z}))\big),\\
H^p\big(\Gamma_{cpt}(X,\mathcal{F}^{n,*})\big) &\overset{[\pi^*_p]}{\longrightarrow}& H^p\big(\Gamma_{cpt}(M,\mathcal{C}^{n,*}_\sigma(L_{|Z|-Z}))\big).
\end{eqnarray*}
In both cases, $\coker [\pi^*_p] = \Gamma\big(X,R^p\pi_* (\mathcal{K}_M\otimes\OO(|Z|-Z))\big)$ for $p\geq 1$,
and the $[\pi_0^*]$ are isomorphisms.
\end{thm}

Combining this with Theorem \ref{thm:Ks},
the well-known fact that $(\mathcal{C}^{n,*}_\sigma(L_{|Z|-Z}),\dq_w)$ is a fine resolution of $\mathcal{K}_M\otimes\OO(|Z|-Z)$,
and that $(\mathcal{F}^{n,*},\dq_s)$ is a fine resolution of $\mathcal{K}_X^s$ (Theorem \ref{thm:exactness2}),
we have completed the proof of Theorem \ref{thm:main1}.


\bigskip

\section{Relation to $L^2$-Dolbeault cohomology}\label{sec:dolbeault}

\subsection{$L^2$-Serre duality.}

We shall shortly recall the use of $L^2$-Serre duality as it was introduced in \cite{PS1} and \cite{eins}.
Let $N$ be a Hermitian complex manifold of dimension $n$.
Let 
$$\dq_{cpt}: A^{p,q}_{cpt}(N) \rightarrow A^{p,q+1}_{cpt}(N)$$
be the $\dq$-operator on smooth forms with compact support in $N$.
Then we denote by
$$\dq_{max}: L^{p,q}(N) \rightarrow L^{p,q+1}(N)$$
the maximal and by
$$\dq_{min}: L^{p,q}(N) \rightarrow L^{p,q+1}(N)$$
the minimal closed Hilbert space extension of the operator $\dq_{cpt}$
as densely defined operator from $L^{p,q}(N)$ to $L^{p,q+1}(N)$.
Let $H^{p,q}_{max}(N)$ be the $L^2$-Dolbeault cohomology on $N$ with respect
to the maximal closed extension $\dq_{max}$, i.e. the $\dq$-operator in the sense of distributions on $N$,
and $H^{p,q}_{min}(N)$ the $L^2$-Dolbeault cohomology with respect to the minimal closed
extension $\dq_{min}$.
Then, $L^2$-Serre duality can be formulated as follows
(see \cite{PS1}, Proposition 1.3 or \cite{eins}, Theorem 2.3):

\begin{thm}\label{thm:l2duality1}
Let $N$ be a Hermitian complex manifold of dimension $n$ and $0\leq p,q\leq n$.
Assume that the $\dq$-operators in the sense of distributions
\begin{eqnarray}
\dq_{max}:&& L^{p,q-1}(N) \rightarrow L^{p,q}(N),\label{eq:dqmax1}\\
\dq_{max}:&& L^{p,q}(N) \rightarrow L^{p,q+1}(N)\label{eq:dqmax2}
\end{eqnarray}
both have closed range (with the usual assumptions for $q=0$ or $q=n$). Then there exists a non-degenerate pairing
$$\{\cdot,\cdot\}: H^{p,q}_{max}(N) \times H^{n-p,n-q}_{min}(N) \rightarrow \C$$
given by
$$\{\eta,\psi\}:=\int_N \eta\wedge\psi.$$
\end{thm}

\begin{proof}
The proof follows by standard arguments (representation of cohomology groups by harmonic representatives)
from the fact that the operators \eqref{eq:dqmax1}, \eqref{eq:dqmax2} have closed range exactly if their $L^2$-adjoints,
the
operators
\begin{eqnarray*}
\dq_{min}:&& L^{n-p,n-q}(N) \rightarrow L^{n-p,n-q+1}(N),\\
\dq_{min}:&& L^{n-p,n-q-1}(N) \rightarrow L^{n-p,n-q}(N)
\end{eqnarray*}
and their $L^2$-adjoints all have closed range (see \cite{PS1}, Proposition 1.3, or \cite{eins}, Theorem 2.3).
\end{proof}

Note that the theorem remains valid for forms with values in Hermitian vector bundles
(where we have to incorporate the duality between the bundle and its dual bundle).
The closed range condition is satisfied e.g. in the following situation that we need to consider
in the present paper:

\begin{thm}\label{thm:l2duality2}
Let $X$ be a Hermitian complex space of pure dimension $n$ with only isolated singularities
and $\Omega\subset\subset X$ a domain which is either compact (without boundary) or has strongly pseudoconvex boundary
which does not intersect the singular set, $b\Omega\cap \Sing X=\emptyset$. Let $0\leq q\leq n$ and $\Omega^*:=\Omega-\Sing X$.

Then the $\dq$-operators in the sense of distributions
\begin{eqnarray}
\dq_{max}:&& L^{0,q-1}(\Omega^*) \rightarrow L^{0,q}(\Omega^*),\label{eq:dqmax3}\\
\dq_{max}:&& L^{0,q}(\Omega^*) \rightarrow L^{0,q+1}(\Omega^*)\label{eq:dqmax4}
\end{eqnarray}
both have closed range and there exists a non-degenerate pairing
$$\{\cdot,\cdot\}: H^{0,q}_{max}(\Omega^*) \times H^{n,n-q}_{min}(\Omega^*) \rightarrow \C$$
given by
$$\{\eta,\psi\}:=\int_{\Omega^*} \eta\wedge\psi.$$
\end{thm}

\begin{proof}
The operators \eqref{eq:dqmax3} and \eqref{eq:dqmax4} have closed range by \cite{OR}, Theorem 1.1.
So, the statement follows immediately from Theorem \ref{thm:l2duality1}.

We remark that the closed range condition can be also deduced from finite-dimensionality of
the corresponding cohomology groups (see \cite{HL}, Appendix 2.4), so that \cite{OR}, Theorem 1.1
is nonessential.
\end{proof}

\smallskip
\subsection{Extension of $H^{p,q}_{max}$-cohomology classes}\label{subsec:extension}

We consider the following situation: Let $X$ be a Hermitian complex space
(of pure dimension $n$) and $\Omega \subset\subset X$ a domain with
smooth strongly pseudoconvex boundary which does not intersect the singular set, i.e. $b\Omega\cap \Sing X=\emptyset$.
Suppose that $\rho: U \rightarrow \R$ is a smooth strictly plurisubharmonic defining function (for $\Omega$)
on a neighborhood $U \subset \Reg X$ of $b\Omega$. For $\epsilon>0$ small enough, let
$$\Omega_\epsilon := \Omega \cup \{z\in U: \rho(z)<\epsilon\}.$$
The purpose of this section is to show that for $0\leq p\leq n$, $q\geq 1$ the natural restriction
$$r: H^{p,q}_{max}(\Omega^*_\epsilon) \rightarrow H^{p,q}_{max}(\Omega^*)$$
is surjective for $\epsilon>0$ small enough 
by use of Grauert's bump method
(where $\Omega^*=\Omega-\Sing X$ and $\Omega^*_\epsilon=\Omega_\epsilon-\Sing X$).

Since there are no singularities in the neighborhood $U$ of the boundary $b\Omega$,
we can use the usual bumping procedure as it is described for example in \cite{LM}, Chapter IV.7.
We only have to make sure that $\Omega - U$ is irrelevant for the procedure.
But this is in fact the case as the bumping procedure looks as follows:

Let $B_0=\Omega$. Then there exists $\epsilon>0$, an integer $t$ and smoothly bounded strongly pseudoconvex domains
$B_1$, ..., $B_t=\Omega_\epsilon$ with $B_{j-1} \subset B_j$ for $j=1, ..., t$
such that the complements $\o{B_j}-B_{j-1}$ are compactly contained in patches $U_j\subset\subset U$
which are biholomorphic to strictly convex smoothly bounded domains in $\C^n$ (we assume that a neighborhood of $\o{U_j}$
is biholomorphic to an open set in $\C^n$).
We fix a certain $j\in\{1, ..., t\}$.
The situation can be arranged such that there exists a (small) neighborhood $V$ of $bU_j\cap b B_{j-1}$,
a strongly pseudoconvex domain $D$ with smooth boundary such that
$$D\subset B_{j-1}\cap U_j\ ,\ \ D-V=(B_{j-1}\cap U_j) -V,$$
and $D$ is biholomorphically equivalent to a strictly convex domain with smooth boundary in $\C^n$,
and that moreover there exists a smooth cut-off function $\psi\in C^\infty_{cpt}(X)$ such that $\supp \psi \subset\subset U_j$,
$\psi \equiv 1$ in a neighborhood of $\o{B_j}-B_{j-1}$ and
$$\supp\psi \cap \o{bD\cap B_{j-1}} = \emptyset.$$
Then, one step in the bumping procedure can be accomplished as follows (see \cite{LM}, Lemma 7.2).
Let $0\leq p\leq n$, $q\geq 1$ and $[\phi]\in H^{p,q}_{max}(B_{j-1})$ be represented by $\phi\in L^{p,q}(B_{j-1})$.
Then there exists $g'\in L^{p,q-1}(D)$ such that
$\dq_{max} g' = \phi$ on $D$.
We set $g:=\psi g'$ on $D$ and $g\equiv 0$ on $B_{j-1}-D$. Then $g\in L^{p,q-1}(B_{j-1})\cap \Dom\dq_{max}$.
Now, we can set
$$\phi' := \left\{\begin{array}{ll}
\phi-\dq_{max} g & \mbox{ on } B_{j-1},\\
0 & \mbox{ on } B_j-B_{j-1}.
\end{array}\right.$$
Then $\phi'\in L^{p,q}(B_j)$ and $\dq_{max}\phi'=0$ such that $\phi'$ defines a class $[\phi']\in H^{p,q}_{max}(B_j)$
with
$$[\phi']|_{B_{j-1}} = [\phi'|_{B_{j-1}}] = [\phi] \in H^{p,q}_{max}(B_{j-1}).$$
An induction over $j=1, ..., t$ shows:

\begin{lem}\label{lem:surjective}
Let $X$ be a Hermitian complex space (of pure dimension $n$) and $\Omega\subset\subset X$
a domain with strongly pseudoconvex smooth boundary that does not intersect the singular set, $b\Omega\cap \Sing X=\emptyset$.
Let $0\leq p\leq n$ and $q\geq 1$.

Then there exists a strongly pseudoconvex smoothly bounded domain $\Omega_\epsilon$
with $\Omega\subset\subset \Omega_\epsilon$ such that the natural restriction map
$$r: H^{p,q}_{max}(\Omega_\epsilon^*) \rightarrow H^{p,q}_{max}(\Omega^*)$$
is surjective (where $\Omega^*=\Omega-\Sing X$ and $\Omega_\epsilon^*=\Omega_\epsilon-\Sing X$).
\end{lem}

The dual statement (according to Theorem \ref{thm:l2duality2}) reads as:

\begin{lem}\label{lem:injective}
Let $X$ be a Hermitian complex space of pure dimension $n$ with only isolated singularities
and $\Omega\subset\subset X$ a domain with strongly pseudoconvex boundary
which does not intersect the singular set, $b\Omega\cap \Sing X=\emptyset$. 
Let $1\leq q \leq n$, and $\Omega_\epsilon$ chosen according to Lemma \ref{lem:surjective}.
Then the natural inclusion map
$$i: H^{0,n-q}_{min}(\Omega^*) \rightarrow H^{0,n-q}_{min}(\Omega_\epsilon^*)$$
is injective.
\end{lem}

\begin{proof}
Note that the map $i$ is defined as follows. For $\psi\in L^{r,s}(\Omega^*)\cap \Dom\dq_{min}$
let $\wt{\psi}$ be the trivial extension by $0$ to $\Omega_\epsilon^*$.
Then $\wt{\psi}\in L^{r,s}(\Omega_\epsilon^*)\cap \Dom\dq_{min}$. So, the map $i$ is given by
the assignment $[\psi]\mapsto [\wt{\psi}]$.

Now then, let $[\psi]\in H^{0,n-q}_{min}(\Omega^*)$ such that 
$$i[\psi]=[\wt{\psi}] =0 \in H^{0,n-q}_{min}(\Omega_\epsilon^*).$$
By Theorem \ref{thm:l2duality2}, this means nothing else but
\begin{eqnarray}\label{eq:zero1}
\int_{\Omega_\epsilon^*} \eta\wedge \wt{\psi} = 0
\end{eqnarray}
for all $[\eta]\in H^{0,q}_{max}(\Omega_\epsilon^*)$.

Let $[\phi]\in H^{0,q}_{max}(\Omega^*)$. By Lemma \ref{lem:surjective}, there exists a class $[\phi']\in H^{0,q}_{max}(\Omega_\epsilon^*)$
such that 
$$r[\phi'] = [\phi']|_{\Omega^*} = [\phi'|_{\Omega^*}] = [\phi] \in H^{0,q}_{max}(\Omega^*).$$
By \eqref{eq:zero1}, this yields
\begin{eqnarray}
\int_{\Omega^*} \phi\wedge\psi = \int_{\Omega^*} \phi'\wedge \psi = \int_{\Omega_\epsilon^*} \phi'\wedge \wt{\psi} =0
\end{eqnarray}
for all $[\phi]\in H^{0,q}_{max}(\Omega^*)$ which (by use of Theorem \ref{thm:l2duality2}) means nothing else
but $[\psi]=0 \in H^{n,n-q}_{min}(\Omega^*)$.
\end{proof}

\smallskip
\subsection{Exceptional sets}\label{subsec:exceptional}

We need some well-known facts about exceptional sets.
Here, we adopt the presentation from \cite{OvVa4}, section 3.1.

Let $X$ be a complex space. A compact nowhere discrete, nowhere dense analytic set $A\subset X$ is an exceptional set
(in the sense of Grauert \cite{G}, §2.Definition 3) if there exists a proper, surjective map $\pi: X\rightarrow Y$ such that $\pi(A)$
is discrete, $\pi: X-A \rightarrow Y -\pi(A)$ is biholomorphic and for every open set $D\subset Y$ the map
$\pi^*: \Gamma(D,\OO_Y)\rightarrow \Gamma(\pi^{-1}(D),\OO_X)$ is surjective.

\begin{thm}\label{thm:grauert}
{(Grauert \cite{G}, §2.Satz 5)}
Let $X$ be a complex space and $A\subset X$ a nowhere discrete compact analytic set.
Then $A$ is an exceptional set exactly if there exists a strongly pseudoconvex neighborhood $U\subset\subset X$ 
of $A$ such that $A$ is the maximal compact analytic subset of $U$.
\end{thm}

\begin{thm}\label{thm:laufer}
{(Laufer \cite{La}, Lemma 3.1)}
Let $\pi: U\rightarrow Y$ exhibit $A$ as exceptional set in $U$ with $Y$ a Stein space.
If $V\subset U$ with $V$ a holomorphically convex neighborhood of $A$ and $\mathcal{F}$
is a coherent analytic sheaf on $U$, then the restriction map $\rho: H^{i}(U,\mathcal{F}) \rightarrow H^{i}(V,\mathcal{F})$
is an isomorphism for $i\geq 1$.
\end{thm}

\smallskip
\subsection{$L^2$-$\dq_{min}$-Representation of $H^q_{cpt}(\Omega,\mathcal{K}_X^s)$}

The purpose of this subsection is to represent the cohomology with compact support of $\mathcal{K}_X^s$
in terms of $L^2$-$\dq_{min}$-cohomology groups.
This can be done on strongly pseudoconvex domains.

For this, we return to the situation of Subsection \ref{subsec:extension},
i.e. 
let $X$ be a Hermitian complex space of pure dimension $n$ with only isolated singularities
and $\Omega \subset\subset X$ a domain with
smooth strongly pseudoconvex boundary which does not intersect the singular set
and $\rho: U \rightarrow \R$ a smooth strictly plurisubharmonic defining function (for $\Omega$)
on a neighborhood $U \subset \Reg X$ of $b\Omega$. For $\epsilon>0$ small enough,
$$\Omega_\epsilon = \Omega \cup \{z\in U: \rho(z)<\epsilon\}.$$
Note that a neighborhood of $\o{\Omega_\epsilon - \Omega}$ does not contain singularities. We have:

\begin{lem}\label{lem:inclusion2}
Let $0 < q < n$. In the situation above, the natural injection
$$i_q: H^q(\Gamma_{cpt}(\Omega,\mathcal{F}^{n,*})) \rightarrow H^q(\Gamma_{cpt}(\Omega_\epsilon,\mathcal{F}^{n,*}))$$
is an isomorphism.
\end{lem}

\begin{proof}
The mapping $i_q$ is induced by extending sections in $\Gamma_{cpt}(\Omega,\mathcal{F}^{n,*})$
trivially to sections in $\Gamma_{cpt}(\Omega_\epsilon,\mathcal{F}^{n,*})$.
Let $\pi: M\rightarrow X$ be a resolution of singularities as in Theorem \ref{thm:main0},
$$\mathcal{R}^q:=R^q \pi_* \big( \mathcal{K}_M \otimes \OO(|Z|-Z)\big),$$
$D:=\pi^{-1}(\Omega)$, $D_\epsilon:=\pi^{-1}(\Omega_\epsilon)$,
and consider the exact diagram
$$
\begin{xy}
  \xymatrix{
      0 \ar[r] & H^q(\Gamma_{cpt}(\Omega,\mathcal{F}^{n,*})) \ar[r]^{\pi^*} \ar[d]^{i_q}    
&   H^q(\Gamma_{cpt}(D,\mathcal{C}^{n,*}_\sigma(L_{|Z|-Z}))) \ar[r]^{\ \ \ \ \ \ \ \ \ \ s} \ar[d]^{j_q} & \Gamma(\Omega,\mathcal{R}^q) \ar[r] \ar[d]^{\cong}& 0  \\
      0 \ar[r] & H^q(\Gamma_{cpt}(\Omega_\epsilon,\mathcal{F}^{n,*})) \ar[r]^{\pi^*}             
&   H^q(\Gamma_{cpt}(D_\epsilon,\mathcal{C}^{n,*}_\sigma(L_{|Z|-Z}))) \ar[r]^{\ \ \ \ \ \ \ \ \ \ s}  & \Gamma(\Omega_\epsilon,\mathcal{R}^q) \ar[r] & 0.
  }
\end{xy}$$
This diagram is commutative because the maps $\pi^*$ are induced by pullback of forms under $\pi: M\rightarrow X$,
$i_q$ and $j_q$ are induced by the trivial extension of forms, 
and the maps $s$ are induced by the residue class maps $\Gamma(\Omega,\mathcal{K}^q) \rightarrow \Gamma(\Omega,\mathcal{R}^q)$
and $\Gamma(\Omega_\epsilon,\mathcal{K}^q) \rightarrow \Gamma(\Omega_\epsilon,\mathcal{R}^q)$, respectively.
The vertical arrow on the right-hand side of the diagramm is an isomorphism because there are no singularities
in a neighborhood of $\o{\Omega_\epsilon-\Omega}$.

But $j_q$ is also an isomorphism. This follows from Theorem \ref{thm:laufer} as follows.
The map
$$j_q: H^q(\Gamma_{cpt}(D,\mathcal{C}^{n,*}_\sigma(L_{|Z|-Z}))) \rightarrow H^q(\Gamma_{cpt}(D_\epsilon,\mathcal{C}^{n,*}_\sigma(L_{|Z|-Z})))$$
is dual to the natural restriction map
$$\rho_{n-q}: H^{n-q}(\Gamma(D_\epsilon, \mathcal{C}^{0,*}_\sigma(L_{Z-|Z|}))) \rightarrow H^{n-q}(\Gamma(D,\mathcal{C}^{0,*}_\sigma(L_{Z-|Z|}))).$$
Since $\mathcal{C}^{0,*}_\sigma(L_{Z-|Z|})$ is a fine resolution of $\OO(Z-|Z|)$,
we can apply Theorem \ref{thm:laufer} to $\OO(Z-|Z|)$ and conclude that $\rho_{n-q}$ is an isomorphism for $q <n$.
Since $D$ and $D_\epsilon$ are strongly pseudoconvex subsets of a complex manifolds,
we can apply Serre duality, and obtain that $j_q$ is also an isomorphism for $q<n$.
But then $i_q$ is an isomorphism as well by commutativity of the diagram.
\end{proof}

From this and Lemma \ref{lem:injective}, we deduce:

\begin{thm}\label{thm:l2representation}
Let $X$ be a Hermitian complex space of pure dimension $n$ with only isolated singularities
(s.t. there exists a representation as in Theorem \ref{thm:main1}, $\mathcal{K}_X^s=\pi_*\big( \mathcal{K}_M\otimes\OO(Z-|Z|)\big)$)
and $\Omega\subset\subset X$ a domain with strongly pseudoconvex boundary
which does not intersect the singular set, $b\Omega\cap \Sing X=\emptyset$. 

Let $0\leq q < n$.
Then the natural inclusion map
$$\iota: H^q(\Gamma_{cpt}(\Omega,\mathcal{F}^{n,*})) \rightarrow H^{n,q}_{min}(\Omega^*)$$
is an isomorphism.
\end{thm}

\begin{proof}
Let $\Omega_\epsilon$ be a strongly pseudoconvex neighborhood of $\Omega$ as in Lemma \ref{lem:injective}
and Lemma \ref{lem:inclusion2}. 
If $q \geq 1$,
we consider the sequence of trivial inclusions
\begin{eqnarray*}
H^q(\Gamma_{cpt}(\Omega,\mathcal{F}^{n,*})) \overset{\iota}{\rightarrow} H^{n,q}_{min}(\Omega^*)
\overset{j}{\rightarrow} H^q(\Gamma_{cpt}(\Omega_\epsilon,\mathcal{F}^{n,*}))
\overset{k}{\rightarrow} H^{n,q}_{min}(\Omega_\epsilon^*).
\end{eqnarray*}
Now then, $j\circ \iota$ is an isomorphism as it is just the map $i_q$ from Lemma \ref{lem:inclusion2}.
Hence, $\iota$ must be injective and $j$ is surjective.

On the other hand, $k\circ j$ is injective as it is just the map $i$ from Lemma \ref{lem:injective}.
Hence, $j$ is also injective.

So, the maps $j$ and $j\circ \iota$ both are isomorphisms. This shows that $\iota$ is an isomorphism, as well.

It remains to treat the case $q=0$. But this is trivial for the groups under consideration both vanish.
Let $\pi: M\rightarrow X$ be a resolution of singularities as in Theorem \ref{thm:Ks}.
Consider a form $\phi \in \Gamma_{cpt}(\Omega,\mathcal{F}^{n,0})$.
Then
$$\pi^*\phi \in \Gamma_{cpt}\big(\pi^{-1}(\Omega), \mathcal{K}_M\otimes\OO(|Z|-Z)\big) = 0$$
since holomorphic $n$-forms with compact support in a non-compact manifold must vanish.
On the other hand, let $\phi\in H^{n,0}_{min}(\Omega^*)$.
Then $\phi$ can be trivially extended to a $\dq_s$-closed form $\wt{\phi}$ with compact support,
$$\wt{\phi} \in \Gamma_{cpt}(\Omega_\epsilon,\mathcal{F}^{n,0}),$$
so that $\phi=0$ as before.
\end{proof}

\subsection{$L^2$-version of Theorem \ref{thm:main1}}

A slight modification of the proof of Theorem \ref{thm:l2representation}
also gives an $L^2$-version of Theorem \ref{thm:main1}.
For this, we also need the well-known:

\begin{lem}\label{lem:l2inclusion3}
Let $(M,\sigma)$ be a Hermitian complex manifold of dimension $n$ and let $G\subset\subset M$ be a strongly pseudoconvex
smoothly bounded domain. Assume that $Z$ is a divisor with support in $G$ and denote by $L_{|Z|-Z}$
a Hermitian holomorphic line bundle as in Section \ref{sec:main1} (the proof of Theorem \ref{thm:main1}).
Then the natural inclusion (trivial extension of forms) induces a natural isomorphism
$$\iota_{cpt}: H^q(\Gamma_{cpt}(G,\mathcal{C}^{n,*}_\sigma(L_{|Z|-Z}))) \rightarrow H^{n,q}_{min}(G,L_{|Z|-Z}))$$
for all $0\leq q<n$.
\end{lem}

\begin{proof}
By use of Serre duality and the $L^2$-duality Theorem \ref{thm:l2duality1} for forms with values in Hermitian line
bundles (see \cite{eins}, Theorem 2.3), it is enough to prove that the natural map
\begin{eqnarray}\label{eq:l2l2loc}
H^{0,n-q}_{max}(G,L_{Z-|Z|}) \rightarrow H^{n-q}(\Gamma(G,\mathcal{C}^{0,*}_\sigma(L_{Z-|Z|})))
\end{eqnarray}
is an isomorphism.
But this is the well-known equivalence of $L^2$- and $L^2_{loc}$-cohomology on strongly pseudoconvex domains.\footnote{
One can prove the isomorphy \eqref{eq:l2l2loc} by combining Lemma \ref{lem:surjective} and Theorem \ref{thm:laufer}.}
The map \eqref{eq:l2l2loc} is induced by the natural inclusion $L^{0,n-q}_\sigma(G,L_{Z-|Z|}) \rightarrow \mathcal{L}^{0,n-q}_\sigma(G,L_{Z-|Z|})$
of $L^2$- into $L^2_{loc}$-forms on $G$, and $(\mathcal{C}^{0,*}_\sigma(L_{Z-|Z|}),\dq_w)$ is the usual fine $L^2_{loc}$-resolution of $\OO(Z-|Z|)$. 
\end{proof}

Combining Theorem \ref{thm:l2representation} with Lemma \ref{lem:l2inclusion3}, we obtain the following 
$L^2$-version of Theorem \ref{thm:main1}:

\begin{thm}\label{thm:main15}
In the situation of Theorem \ref{thm:main1}, let $\Omega\subset\subset X$ be a domain with
strongly pseudoconvex smooth boundary that does not intersect the singular set, $b\Omega\cap\Sing X=\emptyset$.
Let $0\leq p < n$, $\wt{\Omega}:=\pi^{-1}(\Omega)$ and $\Omega^*=\Omega-\Sing X$.

Then the pull-back of forms $\pi^*$ induces a natural injective homomorphism
\begin{eqnarray*}
h_p: H^{n,p}_{min}(\Omega^*) \longrightarrow H^{n,p}_{min}(\wt{\Omega},L_{|Z|-Z})
\end{eqnarray*}
with $\coker h_p = \Gamma(\Omega,R^p\pi_* (\mathcal{K}_M\otimes\OO(|Z|-Z))$ if $p\geq 1$,
and $h_0$ is an isomorphism.
\end{thm}

\begin{proof}
Consider the diagram of natural mappings
$$
\begin{xy}
  \xymatrix{
H^p(\Gamma_{cpt}(\Omega,\mathcal{F}^{n,*})) \ar[r]^{[\pi^*_p]} \ar[d]^{\iota} & H^p(\Gamma_{cpt}(\wt{\Omega},\mathcal{C}^{n,*}_\sigma(L_{|Z|-Z}))) \ar[d]^{\iota_{cpt}}\\
H^{n,p}_{min}(\Omega^*) \ar[r]^{h_p} & H^{n,p}_{min}(\wt{\Omega},L_{|Z|-Z}),}
\end{xy}$$
where $[\pi^*_p]$ is the map from Theorem \ref{thm:main0},
$\iota$ is the map from Theorem \ref{thm:l2representation},
and $\iota_{cpt}$ is the map from Lemma \ref{lem:l2inclusion3}.
Since all these maps are induced by natural inclusion 
or pull-back of forms, the diagram is commutative.
But $\iota$ and $\iota_{cpt}$ are isomorphisms and $[\pi_p^*]$ is injective
with $\coker [\pi_p^*] = \Gamma(\Omega,R^p\pi_* (\mathcal{K}_M\otimes\OO(|Z|-Z))$ if $p\geq 1$
and an isomorphism for $p=0$.
That implies the assertion.
\end{proof}

The situation is much easier on compact complex spaces (where we can also include the case $p=n$):

\begin{thm}\label{thm:mainc}
In the situation of Theorem \ref{thm:main1}, assume that $X$ is compact.
Let $0\leq p\leq n$.
Then the pull-back of forms $\pi^*$ induces a natural injective homomorphism
\begin{eqnarray*}
h_p: H^{n,p}_{min}(X-\Sing X) \longrightarrow H^{n,p}(M,L_{|Z|-Z})
\end{eqnarray*}
with $\coker h_p = \Gamma(X,R^p\pi_* (\mathcal{K}_M\otimes\OO(|Z|-Z))$ if $p\geq 1$, and $h_0$ is an isomorphism.
\end{thm}

\begin{proof}
Since $X$ is compact, the statement follows directly from Theorem \ref{thm:main0}
because then
\begin{eqnarray*}
H^p(\Gamma_{cpt}(X,\mathcal{F}^{n,*})) &=& H^{n,p}_{min}(X-\Sing X),\\
H^p(\Gamma_{cpt}(M,\mathcal{C}^{n,*}_\sigma(L_{|Z|-Z}))) &=& H^{n,p}(M,L_{|Z|-Z}).
\end{eqnarray*}
for all $0\leq p\leq n$.
\end{proof}

\smallskip
\subsection{Mayer-Vietoris for the $\dq_{min}$-complex}

We will now remove the assumption that $\Omega$ has strongly pseudoconvex boundary from Theorem \ref{thm:main15}
by use of the Mayer-Vietoris sequence for the $\dq_{min}$-complex.

Let $N$ be a Hermitian complex manifold. For an open set $V\subset N$, let
$$C^{p,q}_{min}(V) := L^{p,q}(V) \cap \Dom\dq_{min}$$
be the square-integrable $(p,q)$-forms in the domain of $\dq_{min}$.
We denote by $C^{p,*}_{min}(V)$ the $\dq_{min}$-complex of $(p,q)$-forms
$$C^{p,0}_{min}(V) \overset{\dq_{min}}{\longrightarrow} C^{p,1}_{min}(V) \overset{\dq_{min}}{\longrightarrow}
C^{p,2}_{min}(V) \overset{\dq_{min}}{\longrightarrow} ... $$
Let $U,V$ be two open sets in $N$ with non-empty intersection.
For any $p$, $q$, let 
\begin{eqnarray*}
i_U: && C^{p,q}_{min}(U\cap V) \rightarrow C^{p,q}_{min}(U),\\
i_V: && C^{p,q}_{min}(U\cap V) \rightarrow C^{p,q}_{min}(V),\\
j_U: && C^{p,q}_{min}(U) \rightarrow C^{p,q}_{min}(U\cup V),\\
j_V: && C^{p,q}_{min}(V) \rightarrow C^{p,q}_{min}(U\cup V)
\end{eqnarray*}
be the trivial extension of $(p,q)$-forms.
Then it is clear that the natural sequence of complexes
\begin{eqnarray}\label{eq:soc}
0\rightarrow C^{p,*}_{min}(U\cap V) \overset{(i_U,i_V)}{\longrightarrow}
C^{p,*}_{min}(U)\oplus C^{p,*}_{min}(V) \overset{j_U-j_V}{\longrightarrow}
C^{p,*}_{min}(U\cup V) \rightarrow 0
\end{eqnarray}
is exact on the left and in the middle.
Assume that \eqref{eq:soc} is also exact on the right.
Then the short exact sequence \eqref{eq:soc} gives rise as usually
to the long exact cohomology sequence
\begin{eqnarray*}
0\longrightarrow && H^{p,0}_{min}(U\cap V) \overset{(i_U,i_V)}{\longrightarrow}
H^{p,0}_{min}(U) \oplus H^{p,0}_{min}(V) \overset{j_U-j_V}{\longrightarrow}
H^{p,0}_{min}(U\cup V) \overset{\delta_0}{\longrightarrow} \\
... \overset{\delta_{q-1}}{\longrightarrow} &&H^{p,q}_{min}(U\cap V) \overset{(i_U,i_V)}{\longrightarrow}
H^{p,q}_{min}(U) \oplus H^{p,q}_{min}(V) \overset{j_U-j_V}{\longrightarrow}
H^{p,q}_{min}(U\cup V) \overset{\delta_q}{\longrightarrow}\\
\overset{\delta_q}{\longrightarrow} &&H^{p,q+1}_{min}(U\cap V) \longrightarrow ...
\end{eqnarray*}

We can use this kind of Mayer-Vietoris sequence to prove the following generalized version
of Theorem \ref{thm:main15}:

\begin{thm}\label{thm:main19}
Let $(X,h)$ be a Hermitian complex space of pure dimension $n\geq 2$
with only isolated singularities and $\pi: M\rightarrow X$ a resolution of singularities
with only normal crossings such that $\mathcal{K}_X^s=\pi_*\big( \mathcal{K}_M\otimes \OO(|Z|-Z)\big)$.

Let $0\leq p < n$, $\Omega\subset\subset X$ a relatively compact domain,
$\wt{\Omega}:=\pi^{-1}(\Omega)$, $\gamma=\pi^* h$ and $\Omega^*=\Omega-\Sing X$.
Provide $\wt{\Omega}$ with a (regular) Hermitian metric which is equivalent to $\gamma$
close to the boundary $b\wt{\Omega}$.

Then the pull-back of forms $\pi^*$ induces a natural injective homomorphism
\begin{eqnarray*}
h_p: H^{n,p}_{min}(\Omega^*) \longrightarrow H^{n,p}_{min}(\wt{\Omega},L_{|Z|-Z})
\end{eqnarray*}
with $\coker h_p = \Gamma(\Omega,R^p\pi_* (\mathcal{K}_M\otimes\OO(|Z|-Z))$ if $p\geq 1$,
and $h_0$ is an isomorphism.
\end{thm}

Note that singularities in the boundary $b\Omega$ of $\Omega$ are permitted.

\begin{proof}
Let $Y:=\Omega\cap \Sing X$ be the (finite) set of isolated singularities in $\Omega$
and $E:=|\pi^{-1}(Y)|$ the (compact) exceptional set in $\wt{\Omega}$.
Let $W\subset\subset V$ be two smoothly bounded neighborhoods of $Y$ in $\Omega$ with strongly pseudoconvex boundary.
For this, one can just take the intersection of $X$ with small balls centered at the isolated singularities
(in local embeddings). Note that $W$ and $V$ do not need to be connected.
Let $U:=\Omega - \o{W}$ and $\chi$ be a smooth cut-off function with compact support in $V$ which is identically $1$
in a neighborhood of $\o{W}$. Then $U$ does not contain singularities and $\chi$ is constant
close to $\Sing X$.

Let $V^*=V-\Sing X$, $\wt{V}:=\pi^{-1}(V)$, $\wt{U}:=\pi^{-1}(U)$ and $\wt{\chi}=\pi^*\chi$.
Choose any (regular) Hermitian metric $\sigma'$ on $M$.
Then $\sigma:=\wt{\chi} \sigma' + (1-\wt{\chi}) \gamma$ is a (regular) Hermitian metric on $\wt{\Omega}$
which is equivalent to $\gamma$ close to $b\wt{\Omega}$.
All such metrics on $\wt{\Omega}$ are equivalent.
In the following, we work with the two Hermitian complex manifolds $(\Omega^*,h)$ and $(\wt{\Omega},\sigma)$.

Let $L:=L_{|Z|-Z}\rightarrow M$ be the holomorphic line bundle as above,
carrying an arbitrary Hermitian metric. This metric can be arbitrary because $\wt{\Omega}$
is relatively compact in $M$ and we can restrict our attention to $Z:=\pi^{-1}(\Omega\cap \Sing X)$.
This setting avoids problems with singularities in the boundary of $\Omega$.

In this situation, the short sequence \eqref{eq:soc} is exact.
The reason is as follows. Let $\omega\in C^{p,*}_{min}(U\cup V^*)$ be a form on $U\cup V^*$.
Then $\chi \omega \in C^{p,*}_{min}(V^*)$, $(\chi-1)\omega\in C^{p,*}_{min}(U)$
and $(j_U-j_V)(\chi\omega,(\chi-1)\omega)=\omega$. Thus, there is a long exact Mayer-Vietoris sequence
as above on $U\cup V^*=\Omega^*$.

Analogously, we have a short exact sequence \eqref{eq:soc} of complexes of forms with values in the Hermitian
holomorphic line bundle $L_{|Z|-Z}$ on $\wt{\Omega}=\wt{U}\cup \wt{V}$ (by using the smooth cut-off
functions $\wt{\chi}$ and $1-\wt{\chi}$), and this yields another long exact sequence on $\wt{\Omega}$
for forms with values in $L_{|Z|-Z}$.

This leads to the long exact natural commutative diagram (set $L:=L_{|Z|-Z}$)
$$
\begin{xy}
  \xymatrix{
..  \ar[r] & H^{n,q}_{min}(\wt{U}\cap \wt{V},L) \ar[r] &
H^{n,q}_{min}(\wt{U},L) \oplus H^{n,q}_{min}(\wt{V},L)  \ar[r] & 
H^{n,q}_{min}(\wt{\Omega},L) \ar[r] & ..\\
.. \ar[r] & H^{n,q}_{min}(U\cap V^*) \ar[r] \ar[u]_{h_q^{U\cap V}}&
H^{n,q}_{min}(U)\oplus H^{n,q}_{min}(V^*) \ar[r] \ar[u]_{h_q^U \oplus h_q^V} &
H^{n,q}_{min}(\Omega^*) \ar[r] \ar[u]_{h_q^{U\cup V}} & ..}
\end{xy}$$
where all the vertical maps are induced by pull-back of forms
(see the proof of Theorem \ref{thm:main1}).
By our choice of the open coverings and metrics,
the maps $h_q^{U\cap V}$ and $h_q^U$ are isomorphisms for all $0\leq q \leq n$.
By Theorem \ref{thm:main15}, the $h_q^V$ are injective for $0\leq q < n$ with $\coker h_0^V=0$
and $\coker h_p^V=\Gamma(V,R^p \pi_* (\mathcal{K}_M\otimes\OO(|Z|-Z)))$ for $1\leq p < n$.
Thus, the same holds for the maps $h_q^{U\cup V}$, $0\leq q <n$. But
$$\Gamma(V,R^p \pi_* (\mathcal{K}_M\otimes\OO(|Z|-Z))) = \Gamma(\Omega,R^p \pi_* (\mathcal{K}_M\otimes\OO(|Z|-Z)))$$
for all $1\leq p < n$, and this finishes the proof.
\end{proof}

\bigskip

\section{Proof of Theorem \ref{thm:main2}}\label{sec:main2}

\subsection{Cohomology with support on the exceptional set}\label{subsec:HE}

Let $\pi: M\rightarrow X$ be a resolution of singularities as above
and $E=|\pi^{-1}(\Omega\cap \Sing X)|$ the exceptional set.
For a closed subset $K$ of $M$ and a sheaf of abelian groups $\mathcal{F}$, we denote 
by $H_K^*(M,\mathcal{F})$ the flabby cohomology of $\mathcal{F}$ with support in $K$.
In this section, we are interested in the case where $K$ is the exceptional set $E$.
A nice review of cohomology with support on the exceptional set can be found in \cite{OvVa4}, Section 3.2,
a more extensive treatment in \cite{Ka}.

Let $G\subset X$ be an open set and $\wt{G}:=\pi^{-1}(G)$.
In the following, we may assume that $E\subset \wt{G}$.
Since $X$ has only isolated singularities, there exists a (not connected) smoothly bounded strongly pseudoconvex neighborhood $V$ of $E$
in $\wt{G}$ which exhibits $E$ as exceptional set in $\wt{G}$ (see Section \ref{subsec:exceptional}).
$H^*_E(V,\mathcal{F})=H^*_E(M,\mathcal{F})$ by excision, and so we have natural homomorphisms
$$\gamma_j: H^j_E(M,\mathcal{F}) \rightarrow H_{cpt}^j(V,\mathcal{F}).$$
Then:

\begin{thm}\label{thm:karras}
{(Karras \cite{Ka}, Proposition 2.3)}
If $\mathcal{F}$ is a coherent analytic sheaf on $M$ such that $\mbox{depth}_x \mathcal{F}\geq d$ for all $x\in V-E$,
then
$$\gamma_j: H^j_E(M,\mathcal{F}) \rightarrow H_{cpt}^j(V,\mathcal{F})$$
is an isomorphism for $j<d$.
\end{thm}

On the other hand we have seen that
\begin{eqnarray*}
\Gamma(G,R^p\pi_* \mathcal{K}_M\otimes\OO(|Z|-Z)) = \lim_{\substack{\longrightarrow\\ U}} H^p(\pi^{-1}(U),\mathcal{K}_M\otimes\OO(|Z|-Z))
\end{eqnarray*}
for $p\geq 1$,
where the limit is over open neighborhoods of $\Sing X\cap G$. But then the natural maps
\begin{eqnarray}\label{eq:alphap}
\alpha_p: H^p(V,\mathcal{K}_M\otimes\OO(|Z|-Z)) \rightarrow \Gamma(G,R^p\pi_* \mathcal{K}_M\otimes\OO(|Z|-Z))
\end{eqnarray}
are isomorphisms for $p\geq 1$ by Theorem \ref{thm:laufer}.

By use of Serre duality, there exists a non-degenerate pairing
\begin{eqnarray}\label{eq:serre2}
H^p(V,\mathcal{K}_M\otimes\OO(|Z|-Z)) \times H^{n-p}_{cpt} (V, \OO(Z-|Z|)) \rightarrow \C.
\end{eqnarray}
Since $\mbox{depth} \OO(Z-|Z|) =n$, we can combine Theorem \ref{thm:karras} with
\eqref{eq:alphap} and \eqref{eq:serre2} and obtain:

\begin{thm}\label{thm:dualityE}
In the situation described above, there is a natural non-degenerate pairing
$$\Gamma(G,R^p\pi_* \mathcal{K}_M\otimes\OO(|Z|-Z)) \times H_E^{n-p}(\wt{G},\OO(Z-|Z|)) \rightarrow \C$$
for $1\leq p\leq n$ which is induced by Serre duality and the natural isomorphisms
\begin{eqnarray*}
\alpha_p: && H^p(V,\mathcal{K}_M\otimes\OO(|Z|-Z)) \rightarrow \Gamma(G,R^p\pi_* \mathcal{K}_M\otimes\OO(|Z|-Z)),\\
\gamma_{n-p}: && H^{n-p}_E(M,\OO(Z-|Z|)) \rightarrow H_{cpt}^{n-p}(V,\OO(Z-|Z|).
\end{eqnarray*}
\end{thm}

\smallskip

\subsection{The natural surjection onto $H^{0,q}(\Omega^*)$}
We use the notation
$$\mathcal{R}^p := R^p\pi_* \big(\mathcal{K}_M\otimes\OO(|Z|-Z)\big)$$
for $p\geq 1$ and set $\mathcal{R}^0 = 0$ for ease of notation
but emphasize that this is clearly not the direct image of $\mathcal{K}_M\otimes \OO(|Z|-Z)$.

Assume first that $\Omega$ is strongly pseudoconvex with smooth boundary that does not intersect
the singular set or that $\Omega=X$ is compact.

By Theorem \ref{thm:main19} or Theorem \ref{thm:mainc}, respectively, there is an exact sequence
of natural homomorphisms (induced by pull-back of forms and the residue map)
\begin{eqnarray}\label{eq:dual01}
0 \rightarrow H^{n,p}_{min}(\Omega^*) \overset{h_p}{\longrightarrow} H^{n,p}_{min}(\wt{\Omega},L_{|Z|-Z})
\overset{\pi_p}{\longrightarrow} \Gamma(\Omega,\mathcal{R}^p) \rightarrow 0,
\end{eqnarray}
where $\Omega^*=\Omega-\Sing X$ and $\wt{\Omega}= \pi^{-1}(\Omega)$.

Let $V\subset\subset \wt{\Omega}$ be a strongly pseudoconvex (smoothly bounded) neighborhood of the exceptional set
in $\wt{\Omega}$ as in Section \ref{subsec:HE}.
For $p\geq 1$,
the map $\pi_p$ in \eqref{eq:dual01} factors through $H^p(V,\mathcal{K}_M\otimes\OO(|Z|-Z))$
by use of the isomorphism $\alpha_p$ in Theorem \ref{thm:dualityE}:
\begin{eqnarray}\label{eq:dual02}
\pi_p: H^{n,p}_{min}(\wt{\Omega},L_{|Z|-Z}) \overset{r_p}{\longrightarrow}
H^p(V,\mathcal{K}_M\otimes\OO(|Z|-Z)) \overset{\alpha_p}{\longrightarrow} \Gamma(\Omega,\mathcal{R}^p),
\end{eqnarray}
where $r_p$ is the natural restriction map. Ignore this factorization in the case $p=0$.

Note that all the cohomology groups in \eqref{eq:dual01} and \eqref{eq:dual02}
are finite-dimensional vector spaces. For any such vector space $G$, we denote the dual space of
(continuous) linear mappings to $\C$ by
$$G^* := \Hom(G,\C).$$
By finite-dimensionality,
the contravariant functor $G\mapsto G^*$ is not only left-exact, but also right-exact.
So, \eqref{eq:dual01} yields the natural exact sequence
\begin{eqnarray}\label{eq:dual03}
0  \rightarrow \Gamma(\Omega,\mathcal{R}^p)^* \overset{\pi_p^*}{\longrightarrow}
H^{n,p}_{min}(\wt{\Omega},L_{|Z|-Z})^*  \overset{h_p^*}{\longrightarrow}
H^{n,p}_{min}(\Omega^*)^* \rightarrow 0,
\end{eqnarray}
where $\pi_p^*$ factors through $H^p(V,\mathcal{K}_M\OO(|Z|-Z))^*$ by use of the isomorphism $\alpha_p^*$:
\begin{eqnarray}\label{eq:dual04}
\pi_p^*: \Gamma(\Omega,\mathcal{R}^p)^* \overset{\alpha_p^*}{\longrightarrow}
H^p(V,\mathcal{K}_M\otimes\OO(|Z|-Z))^* \overset{r_p^*}{\longrightarrow}
H^{n,p}_{min}(\wt{\Omega},L_{|Z|-Z})^*.
\end{eqnarray}
Again, ignore this factorization if $p=0$.
The situation is quite comfortable since all the groups under consideration are of finite dimension.
Anyway, the argument remains valid if only the corresponding $\dq$-operators all have closed range.
Then, all the spaces under consideration inherit a natural locally convex topology (see \cite{Rd}, Theorem 1.41),
and then the functor $G\mapsto G^*$ is right-exact by the Hahn-Banach Theorem (see \cite{Rd}, Theorem 3.6).

Now, consider the commutative diagram of natural mappings (let $q=n-p$)
$$
\begin{xy}
  \xymatrix{
0  \ar[r] &  \Gamma(\Omega,\mathcal{R}^p)^* \ar[r]^{\pi_p^*} &
H^{n,p}_{min}(\wt{\Omega},L_{|Z|-Z})^*  \ar[r]^{h_p^*} & 
H^{n,p}_{min}(\Omega^*)^* \ar[r] & 0\\
0 \ar[r] & H^{q}_E(\wt{\Omega},\OO(Z-|Z|)) \ar[r]^{i_q} \ar[u]^{\cong}&
H^{0,q}_{max}(\wt{\Omega},L_{Z-|Z|}) \ar@{.>}[r]^{s_q} \ar[u]^{\cong}&
H^{0,q}_{max}(\Omega^*) \ar[r] \ar[u]^{\cong}& 0,}
\end{xy}$$
where the vertical isomorphisms are induced by the natural non-degenerate pairings
from Theorem \ref{thm:dualityE}, Theorem \ref{thm:l2duality1} for forms with values in holomorphic line bundles
(see \cite{eins}, Theorem 2.3) and Theorem \ref{thm:l2duality1} itself.\footnote{
Here, we use that $\Omega$ is compact or has a strongly pseudoconvex boundary
so that the $\dq$-operators under consideration have closed range and the duality statements can be applied.}
The map $i_q$ is the isomorphism $\gamma_{n-p}$ from Theorem \ref{thm:dualityE} followed
by a natural mapping:
\begin{eqnarray}\label{eq:dual05}
i_q: H^{q}_E(\wt{\Omega},\OO(Z-|Z|)) \overset{\gamma_{n-p}}{\longrightarrow}
H^q_{cpt}(V,\OO(Z-|Z|)) \longrightarrow H^{0,q}_{max}(\wt{\Omega},L_{Z-|Z|}).
\end{eqnarray}
Put $0$ instead of $H^n_E(\wt{\Omega},\OO(Z-|Z|))$ in the case $p=0$ (i.e. if $q=n$).
Now, define 
\begin{eqnarray*}
s_q: H^{0,q}_{max}(\wt{\Omega},L_{Z-|Z|}) \longrightarrow H^{0,q}_{max}(\Omega^*)
\end{eqnarray*}
to be the natural mapping which makes the diagram commutative.
A realization of this mapping on the level of forms can be obtained
by working with harmonic representatives of cohomology classes.

The lower line of the natural commutative diagram must be exact for the upper line is exact.
That proves the second main Theorem \ref{thm:main2} if $\Omega$ is compact or has a strongly pseudoconvex 
smooth boundary without singularities.

The general statement follows exactly as in the proof of Theorem \ref{thm:main19}
by use of the Mayer-Vietoris sequence for the $\dq_{max}$-cohomology.

\bigskip
We may add a few words about the commutative diagram. 
Consider a cohomology class
$[a]\in H^q_E(\wt{\Omega},\OO(Z-|Z|))$.
Then $[a]$ is represented by a $(0,q)$-form with compact support in $V$:
$$a\in \Gamma_{cpt}(V,\mathcal{C}^{0,q}_\sigma(L_{Z-|Z|}))$$
and defines a functional in $\Gamma(\Omega,\mathcal{R}^p)^*$ by
$$[a]: [b] \mapsto \int_V a\wedge b,$$
where $b\in \Gamma(V,\mathcal{C}^{n,p}_\sigma(L_{|Z|-Z}))$ is a representative of $[b]\in\Gamma(\Omega,\mathcal{R}^p)$ on $V$.
Since $a$ has compact support in $V$, it defines also the functional $\pi_p^* [a]\in H^{n,p}_{min}(\wt{\Omega},L_{|Z|-Z})^*$ by the assignment
$$\pi_p^*[a]: [c] \mapsto \int_{\wt{\Omega}} a\wedge c,$$
where $c\in\Gamma(\wt{\Omega},\mathcal{C}^{n,p}_\sigma(L_{|Z|-Z}))$ is a representative of $[c]\in H^{n,p}_{min}(\wt{\Omega},L_{|Z|-Z})$.
On the other hand, the class $i_q[a] \in H^{0,q}_{max}(\wt{\Omega},L_{Z-|Z|})$ is also represented by the form $a$.
But then $\pi_p^*[a] = i_q[a]$ in $H^{n,p}_{min}(\wt{\Omega},L_{|Z|-Z})^*$, showing that the left hand square of the diagram is commutative.

Assume that $i_q[a]=0$ in $H^{0,q}_{max}(\wt{\Omega},L_{Z-|Z|})$,
i.e. that there exists a form $f\in\Gamma(\wt{\Omega},\mathcal{C}^{0,q-1}_\sigma(L_{Z-|Z|}))$
such that $\dq_{max} f= a$. Let $\chi$ be a smooth cut-off function with compact support in $V$
which is identically $1$ in a neighborhood of the exceptional set $E$.
Then $[a]=[a-\dq(\chi f)]$ in $H^q_E(\wt{\Omega},\OO(Z-|Z|))$.
But now $a-\dq(\chi f)$ is vanishing in a neighborhood of $E$ meaning that it can be considered as a form
in $\Gamma_{cpt}(V, \mathcal{C}^{0,q}_\sigma)$.
But $H^q_{cpt}(V,\OO) \cong H^{n-q}(V,\mathcal{K}_M)=0$ by Takegoshi's theorem.
So there exists $g\in \Gamma_{cpt}(V,\mathcal{C}^{0,q-1}_\sigma) \subset \Gamma_{cpt}(V,\mathcal{C}^{0,q-1}_\sigma(L_{Z-|Z|}))$
such that $\dq g = a-\dq(\chi f)$.
Thus $[a]=[a-\dq(\chi f)]=0 \in H^q_E(\wt{\Omega},\OO(Z-|Z|))$ showing again that $i_p$ is injective.


\end{document}